\documentclass{article}%
\usepackage{amsfonts}
\usepackage{amsmath}
\usepackage{amssymb}
\usepackage{graphicx}%
\providecommand{\U}[1]{\protect\rule{.1in}{.1in}}
\numberwithin{equation}{section}
\newtheorem{theorem}{Theorem}[section]

\newtheorem{corollary}[theorem]{Corollary}

\newtheorem{lemma}[theorem]{Lemma}

\newtheorem{proposition}[theorem]{Proposition}
\newtheorem{remark}[theorem]{Remark}

\newenvironment{proof}[1][Proof]{\noindent\textbf{#1.} }{\ \rule{0.5em}{0.5em}}

\usepackage{xcolor}

\begin{document}

\title{Second-Order $\Gamma$-Limit for the Cahn--Hilliard Functional with Dirichlet
Boundary Conditions, II}
\author{Irene Fonseca\\Department of Mathematical Sciences,\\Carnegie Mellon University,\\Pittsburgh PA 15213-3890, USA
\and Leonard Kreutz\\Department of Mathematics, Technical University of Munich \\Munich, 80333, Germany
\and Giovanni Leoni\\Department of Mathematical Sciences, \\Carnegie Mellon University, \\Pittsburgh PA 15213-3890, USA}
\maketitle

\begin{abstract}
This paper continues the study of the asymptotic development of order 2 by
$\Gamma$ -convergence of the Cahn--Hilliard \ functional with Dirichlet
boundary conditions initiated in \cite{fonseca-kreutz-leoni2025I}. While in
the first paper, the Dirichlet data are assumed to be well separated from one
of the two wells, here this is no longer the case. In the case where there are
no interfaces, it is shown that there is a transition layer near the boundary
of the domain.

\end{abstract}

\section{Introduction}

In a recent paper \cite{fonseca-kreutz-leoni2025I} we began the study of the
second-order asymptotic development via $\Gamma$-convergence of the Cahn-Hilliard
functional
\begin{equation}
F_{\varepsilon}(u):=\int_{\Omega}(W(u)+\varepsilon^{2}|\nabla u|^{2}%
)\,dx,\quad u\in H^{1}(\Omega), \label{functional cahn-hilliard}%
\end{equation}
subject to the Dirichlet boundary condition%
\begin{equation}
\operatorname*{tr}u=g_{\varepsilon}\quad\text{on }\partial\Omega.
\label{dirichlet boundary conditions}%
\end{equation}
Here $W:\mathbb{R}\rightarrow\lbrack0,\infty)$ is a double-well potential
with
\begin{equation}
W^{-1}(\{0\})=\{a,b\}, \label{wells}%
\end{equation}
$\Omega\subset\mathbb{R}^{N}$ is an open, bounded set with a smooth boundary,
$N\geq2$, and $g_{\varepsilon}\in H^{1/2}(\partial\Omega)$.

We recall that, given a metric space $X$ and a family of functions
$\mathcal{F}_{\varepsilon}:X\rightarrow\lbrack-\infty,\infty]$ for
$\varepsilon>0$, \emph{the asymptotic development of order} $n$ \emph{via
$\Gamma$-convergence}, written as%
\begin{equation}
\mathcal{F}_{\varepsilon}=\mathcal{F}^{(0)}+\varepsilon\mathcal{F}%
^{(1)}+\cdots+\varepsilon^{n}\mathcal{F}^{(n)}+o(\varepsilon^{n}%
),\label{asymptotic development}%
\end{equation}
holds if we can find functions $\mathcal{F}^{(i)}:X\rightarrow\lbrack
-\infty,\infty]$, $i=0,\ldots,n$, such that the functions
\[
\mathcal{F}_{\varepsilon}^{(i)}:=\frac{\mathcal{F}_{\varepsilon}^{(i-1)}%
-\inf_{X}\mathcal{F}^{(i-1)}}{\varepsilon}%
\]
are well-defined and the family $\{\mathcal{F}_{\varepsilon}^{(i)}%
\}_{\varepsilon}$ $\Gamma$-converges to $\mathcal{F}^{(i)}$ as $\varepsilon
\rightarrow0^{+}$ (see \cite{anzellotti-baldo1993} and
\cite{anzellotti-baldo-orlandi1996}). In many cases, the powers
$\varepsilon^{k}$ in the asymptotic development (\ref{asymptotic development})
may be replaced by more general scales, where $\delta_{\varepsilon}^{(i)}>0$
for all $i=1,\ldots,m$ and $\varepsilon>0$, $\delta_{\varepsilon}^{(0)}:=1$
and $\sigma_{\varepsilon}^{(i)}:=\delta_{\varepsilon}^{(i)}/\delta
_{\varepsilon}^{(i-1)}\rightarrow0$ as $\varepsilon\rightarrow0^{+}$ for all
$i=1,\ldots,m$, and the asymptotic expansion takes the form:%
\[
\mathcal{F}_{\varepsilon}=\mathcal{F}^{(0)}+\delta_{\varepsilon}%
^{(1)}\mathcal{F}^{(1)}+\cdots+\delta_{\varepsilon}^{(n)}\mathcal{F}%
^{(n)}+o(\delta_{\varepsilon}^{(n)}),
\]
where the functions $\mathcal{F}_{\varepsilon}^{(i)}$ are defined by
\[
\mathcal{F}_{\varepsilon}^{(i)}:=\frac{\mathcal{F}_{\varepsilon}^{(i-1)}%
-\inf_{X}\mathcal{F}^{(i-1)}}{\sigma_{\varepsilon}^{(i)}}.
\]

The first order asymptotic development of (\ref{functional cahn-hilliard}),
(\ref{dirichlet boundary conditions}) was studied by Owen, Rubinstein, and
Sternberg \cite{owen-rubinstein-sternberg1990} (see also
\cite{cristoferi-gravina2021} and \cite{gazoulis2024}), who proved that the
family of functionals
\[
\mathcal{F}_{\varepsilon}^{(1)}(u)=\left\{
\begin{array}
[c]{ll}%
\int_{\Omega}(\frac{1}{\varepsilon}W(u)+\varepsilon|\nabla u|^{2})\,dx &
\text{if }u\in H^{1}(\Omega),\,\operatorname*{tr}u=g_{\varepsilon}\text{ on
}\partial\Omega,\\
\infty & \text{otherwise in }L^{1}(\Omega),
\end{array}
\right.
\]
$\Gamma$-converges as $\varepsilon\rightarrow0^{+}$ in $L^{1}(\Omega)$ to%
\begin{equation}
\mathcal{F}^{(1)}(u):=\left\{
\begin{array}
[c]{ll}%
C_{W}\operatorname*{P}(\{u=b\};\Omega)+\int_{\partial\Omega}\operatorname*{d}%
\nolimits_{W}(\operatorname*{tr}u,g)\,d\mathcal{H}^{N-1} & \text{if }u\in
BV(\Omega;\{a,b\}),\\
\infty & \text{otherwise in }L^{1}(\Omega),
\end{array}
\right.  \label{firstOrderFormalDefinition}%
\end{equation}
where $\operatorname*{P}(\{u=b\};\Omega)$ is the perimeter of the set
$\{u=b\}$ in $\Omega$, $g_{\varepsilon}\rightarrow g$ in $L^{1}(\partial
\Omega)$, $\operatorname*{d}\nolimits_{W}$ is the geodesic distance determined, to be precise, 
by $W$:
\begin{equation}
\operatorname*{d}\nolimits_{W}(r,s):=\left\{
\begin{array}
[c]{ll}%
2\left\vert \int_{r}^{s}W^{1/2}(\rho)\,d\rho\right\vert  & \text{if }%
r\in\{a,b\}\text{ or }s\in\{a,b\},\\
\infty & \text{otherwise,}%
\end{array}
\right.  \label{distance definition}%
\end{equation}
and
\begin{equation}
C_{W}:=2\int_{a}^{b}W^{1/2}(\rho)\,d\rho. \label{cW definition}%
\end{equation}

In \cite{fonseca-kreutz-leoni2025I}, we studied the second-order asymptotic
expansion of (\ref{functional cahn-hilliard}),
(\ref{dirichlet boundary conditions}) under the hypothesis that the boundary
data $g_{\varepsilon}:\overline{\Omega}\rightarrow\mathbb{R}$ stay away from
one of the two wells $a$, $b$:%
\begin{equation}
a<\alpha_{-}\leq g_{\varepsilon}(x)\leq b \label{bounds g}%
\end{equation}
for all $x\in\overline{\Omega}$, all $\varepsilon\in(0,1)$, and some constant
$\alpha_{-}$. If the constant $\alpha_{-}$ is sufficiently close to $b$, the
only minimizer of $\mathcal{F}^{(1)}$ is the constant function $b$ (see
\cite[Proposition 2.5]{fonseca-kreutz-leoni2025I}). Hence, it is natural to assume that
\begin{equation}
u_{0}\equiv b\quad\text{is the unique minimizer of }\mathcal{F}^{(1)}\text{.}
\label{u0=b}%
\end{equation}
Under this hypothesis, we define%
\begin{align}
&  \mathcal{F}_{\varepsilon}^{(2)}(u):=\frac{\mathcal{F}_{\varepsilon}%
^{(1)}(u)-\min\mathcal{F}^{(1)}}{\varepsilon}\label{F 2 epsilon a<aplha}\\
&  =\int_{\Omega}\left(  \frac{1}{\varepsilon^{2}}W(u)+|\nabla u|^{2}\right)
\,dx-\frac{1}{\varepsilon}\int_{\partial\Omega}\operatorname*{d}%
\nolimits_{W}(b,g)\,d\mathcal{H}^{N-1}\nonumber
\end{align}
if $u\in H^{1}(\Omega)$ and $\operatorname*{tr}u=g_{\varepsilon}$ on
$\partial\Omega$, and $\mathcal{F}_{\varepsilon}^{(2)}(u):=\infty$ otherwise
in $L^{1}(\Omega)$.

The main result in \cite{fonseca-kreutz-leoni2025I} is the following theorem.

\begin{theorem}
Let $\Omega\subset\mathbb{R}^{N}$ be an open, bounded, connected set with a boundary of class $C^{2,d}\ $, $0<d\leq1$. Assume that $W$ satisfies \eqref{W_Smooth}-\eqref{W' three zeroes} and that
$g_{\varepsilon}\in H^{1}(\partial\Omega)$ is such that%
\[
(\varepsilon|\log\varepsilon|)^{1/2}\int_{\partial\Omega}|\nabla_{\tau
}g_{\varepsilon}|^{2}d\mathcal{H}^{N-1}=o(1)\quad\text{as }\varepsilon
\rightarrow0^{+},
\]
and
\[
|g_{\varepsilon}(x)-g(x)|\leq C\varepsilon^{\gamma},\quad x\in\partial\Omega,
\]
for all $\varepsilon\in(0,1)$ and for some constants $C>0$ and $\gamma>1$.
Suppose also that \eqref{u0=b} holds. Then
\begin{equation}
\mathcal{F}^{(2)}(u)=\int_{\partial\Omega}\kappa(y)\int_{0}^{\infty}%
2W^{1/2}(z_{g(y)}(s))z_{g(y)}^{\prime}(s)s\,ds\, d\mathcal{H}^{N-1}(y)
\label{second order}%
\end{equation}
if $u=b$ and $\mathcal{F}^{(2)}(u)=\infty$ otherwise in $L^{1}(\Omega)$, where
$\kappa$ is the mean curvature of $\partial\Omega$ and $z_{\alpha}$ is the
solution to the Cauchy problem
\begin{equation}
\left\{
\begin{array}
[c]{l}%
z_{\alpha}^{\prime}=W^{1/2}(z_{\alpha}),\\
z_{\alpha}(0)=\alpha
\end{array}
\right.  \label{cauchy problem z alpha}%
\end{equation}
with $\alpha=g(y)$.
\end{theorem}

Here, $\nabla_{\tau}$ denotes the tangential gradient.

In the present paper, we relax the bound from below in (\ref{bounds g}) and
allow $g_{\varepsilon}$ to take the value $a$,%
\begin{equation}
a\leq g_{\varepsilon}(x)\leq b, \label{bounds g a=alpha}%
\end{equation}
while still assuming (\ref{u0=b}). We observe that this scenario can only
happen if $\{g=a\}\subseteq\{\kappa\leq0\}$ (see Theorem
\ref{theorem curvature}). If we assume that
\begin{equation}
\{g=a\}\subseteq\{\kappa<0\}, \label{set g=a}%
\end{equation}
then the rescaling (\ref{F 2 epsilon a<aplha}) should be replaced by
\begin{align}
&  \mathcal{F}_{\varepsilon}^{(2)}(u):=\frac{\mathcal{F}_{\varepsilon}%
^{(1)}(u)-\min\mathcal{F}^{(1)}}{\varepsilon|\log\varepsilon|}%
\label{F 2 epsilon a=aplha}\\
&  =\frac{1}{\varepsilon|\log\varepsilon|}\int_{\Omega}\left(  \frac
{1}{\varepsilon}W(u)+\varepsilon|\nabla u|^{2}\right)  \,dx-\frac
{1}{\varepsilon|\log\varepsilon|}\int_{\partial\Omega}\operatorname*{d}%
\nolimits_{W}(b,g)\,d\mathcal{H}^{N-1}\nonumber
\end{align}
if $u\in H^{1}(\Omega)$ and $\operatorname*{tr}u=g_{\varepsilon}$ on
$\partial\Omega$, and $\mathcal{F}_{\varepsilon}^{(2)}(u):=\infty$ otherwise
in $L^{1}(\Omega)$.

The main result of this paper is the following theorem.

\begin{theorem}
\label{theorem main}Let $\Omega\subset\mathbb{R}^{N}$ be an open, bounded,
connected set with boundary of class $C^{2,d}\ $, $0<d\leq1$. Assume that $W$
satisfies \eqref{W_Smooth}-\eqref{W' three zeroes} and that
$g_{\varepsilon}$ satisfy \eqref{bounds g a=alpha},
\eqref{g epsilon smooth a}-\eqref{g epsilon -g bound a}. Suppose also that
\eqref{u0=b} holds. Then
\[
\mathcal{F}^{(2)}(u)=\frac{C_{W}}{2^{1/2}(W^{\prime\prime}(a))^{1/2}}%
\int_{\partial\Omega\cap\{g=a\}}\kappa(y)\,d\mathcal{H}^{N-1}(y)
\]
if $u=b$ and $\mathcal{F}^{(2)}(u)=\infty$ otherwise in $L^{1}(\Omega)$. Here,
$\mathcal{F}^{(2)}$ is defined in \eqref{F 2 epsilon a=aplha}, $\kappa$ is the
mean curvature of $\partial\Omega$, and $C_{W}$ is the constant defined in \eqref{cW definition}.

In particular, if $u_{\varepsilon}\in H^{1}(\Omega)$ is a minimizer of
\eqref{functional cahn-hilliard} subject to the Dirichlet boundary condition
\eqref{dirichlet boundary conditions}, then
\begin{align}
\int_{\Omega}(W(u_{\varepsilon})  &  +\varepsilon^{2}|\nabla u_{\varepsilon
}|^{2})\,dx=\varepsilon\int_{\partial\Omega}\operatorname*{d}\nolimits_{W}%
(b,g)\,d\mathcal{H}^{N-1}\label{sharp bound}\\
&  +\varepsilon^{2}|\log\varepsilon|\frac{C_{W}}{2^{1/2}(W^{\prime\prime
}(a))^{1/2}}\int_{\partial\Omega\cap\{g=a\}}\kappa(y)\,d\mathcal{H}%
^{N-1}(y)+o(\varepsilon^{2}|\log\varepsilon|).\nonumber
\end{align}
\bigskip
\end{theorem}

When the Dirichlet boundary conditions (\ref{dirichlet boundary conditions})
are replaced by the mass constraint
\begin{equation}
\int_{\Omega}u(x)\,dx=m, \label{mass constraint}%
\end{equation}
the first-order asymptotic expansion of the Cahn-Hilliard functional
(\ref{functional cahn-hilliard}) was characterized in \cite{baldo1990},
\cite{fonseca-tartar1989}, \cite{modica-mortola1977}, \cite{modica1987},
\cite{sternberg1988}, while the second order asymptotic expansion was first
proved by the third author and Murray in \cite{leoni-murray2016},
\cite{leoni-murray2019} in dimension $N\geq2$ (see also
\cite{bellettini-nayam-novaga2015}).

As in \cite{leoni-murray2016}, \cite{leoni-murray2019}, our proof relies on
the asymptotic development of order two by $\Gamma$-convergence of the
weighted one-dimensional functional%
\begin{equation}
G_{\varepsilon}(v):=\int_{0}^{T}(W(v(t))+\varepsilon^{2}(v^{\prime}%
(t))^{2})\omega(t)\,dt,\quad v\in H^{1}(I), \label{functional 1d}%
\end{equation}
subject to the Dirichlet boundary conditions
\begin{equation}
v(0)=\alpha_{\varepsilon},\quad v(T)=\beta_{\varepsilon},
\label{dirichlet boundary conditions 1d}%
\end{equation}
where $\omega$ is a smooth positive weight, and
\begin{equation}
a\leq\alpha_{\varepsilon},\,\beta_{\varepsilon}\leq b.
\label{initial values 1d}%
\end{equation}
The key difference in our proof of the  $\Gamma$-liminf inequality is that in
\cite{leoni-murray2016}, \cite{leoni-murray2019}, the authors utilized a
rearrangement technique based on the isoperimetric function to reduce the
functional (\ref{functional cahn-hilliard}) to the one-dimensional weighted
problem. This approach, however, seems to be difficult to implement in this new context except
under trivial boundary conditions. Instead, we adopt techniques from
Sternberg and Zumbrum \cite{sternberg-zumbrun1998} and
Caffarelli and Cordoba \cite{caffarelli-cordoba1995} to analyze the behavior
of minimizers of (\ref{functional cahn-hilliard}) and
(\ref{dirichlet boundary conditions}) near the boundaryy, leveraging slicing
arguments in our study.

This paper is organized as follows. In Section \ref{section 1d functional}, we
characterize the asymptotic development of order two by $\Gamma$-convergence
of the weighted one-dimensional family of functionals $G_{\varepsilon}$
defined in (\ref{functional 1d}). Section \ref{section minimizers} explores
qualitative properties of critical points and minimizers of the functional
\ref{functional cahn-hilliard}. Finally, in Section
\ref{section main theorems}, we prove Theorem \ref{theorem main}.

\section{Preliminaries}

We assume that the double-well potential $W:\mathbb{R}\rightarrow
\lbrack0,\infty)$ satisfies the following hypotheses:%

\begin{align}
 \begin{split} &W\text{ is of class $C^{2,\alpha_{0}}(\mathbb{R})$, }\alpha_{0}%
\in(0,1)\text{, and has precisely two zeros} \\&\text{at } a \text{ and } b, \text{ with } a<b,\end{split}\label{W_Smooth}\\
&  W^{\prime\prime}(a)>0,\quad W^{\prime\prime}(b)>0,\label{WPrime_At_Wells}\\
&  \lim_{s\rightarrow-\infty}W^{\prime}(s)=-\infty,\quad\lim_{s\rightarrow
\infty}W^{\prime}(s)=\infty,\label{WGurtin_Assumption}\\
&  \text{$W^{\prime}$ has exactly 3 zeros at $a,b,c$ with $a<c<b$,}\quad W^{\prime\prime
}(c)<0, \label{W' three zeroes}%
\end{align}
Let
\begin{equation}
a<\alpha_{-}<\min\left\{  c,\frac{a+b}{2}\right\}  \leq\max\left\{
c,\frac{a+b}{2}\right\}  <\beta_{-}<b. \label{alpha and beta minus}%
\end{equation}

\begin{remark}
\label{remark W near b}Since $W\in C^{2}(\mathbb{R})$, $W(a)=W^{\prime}(a)=0$,
$W(b)=W^{\prime}(b)=0$, and $W^{\prime\prime}(a)$, $W^{\prime\prime}(b)>0$,
there exists a constant $\sigma>0$ depending on $\alpha_{-}$ and $\beta_{-}$
such that%
\begin{align}
\sigma^{2}(b-s)^{2}  &  \leq W(s)\leq\frac{1}{\sigma^{2}}(b-s)^{2}%
\quad\text{for all }\alpha_{-}\leq s\leq b+1,\label{W near b}\\
\sigma^{2}(s-a)^{2}  &  \leq W(s)\leq\frac{1}{\sigma^{2}}(s-a)^{2}%
\quad\text{for all }a-1\leq s\leq\beta_{-}. \label{W near a}%
\end{align}

\end{remark}

\begin{proposition}
\label{proposition asymptotic behavior}For $a<\beta<\beta_{-}$, we have
\begin{equation}
\lim_{\varepsilon\rightarrow0^{+}}\frac{\int_{a}^{\beta}\frac{1}%
{(\varepsilon+W(s))^{1/2}}\,ds}{|\log\varepsilon|}=\frac{1}{2^{1/2}%
(W^{\prime\prime}(a))^{1/2}}, \label{limit integral}%
\end{equation}
while for $a<\alpha<b$,%
\[
\lim_{\varepsilon\rightarrow0^{+}}\frac{\int_{\alpha}^{b}\frac{1}%
{(\varepsilon+W(s))^{1/2}}\,ds}{|\log\varepsilon|}=\frac{1}{2^{1/2}%
(W^{\prime\prime}(b))^{1/2}}.
\]
In particular, there exists a constant $C>0$ depending only on $W$ such that%
\begin{equation}
\int_{a}^{b}\frac{1}{(\varepsilon+W(s))^{1/2}}\,ds\leq C|\log\varepsilon|
\label{near b log}%
\end{equation}
for all $0<\varepsilon<\varepsilon_{0}$, where $\varepsilon_{0}>0$ depends
only $W$.
\end{proposition}

\begin{proof}
\textbf{Step 1: }Given $c_{0}>0$, we estimate%
\[
\mathcal{A}:=\int_{a}^{\beta}\frac{1}{(\varepsilon+c_{0}(s-a)^{2})^{1/2}%
}\,ds.
\]
Consider the change of variables $\frac{\varepsilon^{1/2}}{c_{0}^{1/2}}t:= s-a$,
so that $\frac{\varepsilon^{1/2}}{c_{0}^{1/2}}dt=ds$. Then%
\begin{align*}
\mathcal{A}  &  =\frac{1}{c_{0}^{1/2}}\int_{0}^{(\beta-a)c_{0}^{1/2}%
/\varepsilon^{1/2}}\frac{1}{\left(  1+t^{2}\right)  ^{1/2}}\,dt=\frac{1}%
{c_{0}^{1/2}}[\log(t+(t^{2}+1)^{1/2})]_{0}^{(\beta-a)c_{0}^{1/2}%
/\varepsilon^{1/2}}\\
&  =\frac{1}{c_{0}^{1/2}}\log\left(  (\beta-a)c_{0}^{1/2}/\varepsilon
^{1/2}+((\beta-a)^{2}c_{0}/\varepsilon+1)^{1/2}\right) \\
&  =\frac{1}{2c_{0}^{1/2}}|\log\varepsilon|+\frac{1}{c_{0}^{1/2}}\log\left(
(\beta-a)c_{0}^{1/2}+((\beta-a)^{2}c_{0}+\varepsilon)^{1/2}\right)  .
\end{align*}
\textbf{Step 2: }By (\ref{W_Smooth}) and (\ref{WPrime_At_Wells}), given
$0<\eta<<1$, we can find $0<\delta_{\eta}<\alpha_{-}-a$ such that
\[
\frac{1}{2}(1-\eta)W^{\prime\prime}(a)(s-a)^{2}\leq W(s)\leq\frac{1}{2}%
(1+\eta)W^{\prime\prime}(a)(s-a)^{2}%
\]
for all $a\leq s<a+\delta_{\eta}$. Hence,
\begin{equation}
\frac{1}{\left(  \varepsilon+c_{1}(s-a)^{2}\right)  ^{1/2}}\leq\frac
{1}{(\varepsilon+W(s))^{1/2}}\leq\frac{1}{\left(  \varepsilon+c_{2}%
(s-a)^{2}\right)  ^{1/2}} \label{1d 202}%
\end{equation}
for all $a\leq s<a+\delta_{\eta}$, where
\begin{equation}
c_{1}=\frac{1}{2}(1+\eta)W^{\prime\prime}(a),\quad c_{2}=\frac{1}{2}%
(1-\eta)W^{\prime\prime}(a). \label{1d 202a}%
\end{equation}
Write%
\[
\int_{a}^{\beta}\frac{1}{(\varepsilon+W(s))^{1/2}}\,ds=\int_{a}^{a+\delta
_{\eta}}\frac{1}{(\varepsilon+W(s))^{1/2}}\,ds+\int_{a+\delta_{\eta}}^{\beta
}\frac{1}{(\varepsilon+W(s))^{1/2}}\,ds.
\]
Since $\min_{[a+s_{\eta},\beta]}W=w_{0}>0$,%
\begin{equation}
\frac{\int_{a+\delta_{\eta}}^{\beta}\frac{1}{(\varepsilon+W(s))^{1/2}}%
\,ds}{|\log\varepsilon|}\leq\frac{\frac{1}{w_{0}^{1/2}}(b-a)}{|\log
\varepsilon|}\rightarrow0 \label{1d 203}%
\end{equation}
as $\varepsilon\rightarrow0^{+}$.

Using (\ref{1d 202}) with $\beta=a+\delta_{\eta}$ and $c_{0}$ replaced by
$c_{1}$ and $c_{2}$, respectively,
\begin{align*}
&  \frac{1}{2c_{1}^{1/2}}|\log\varepsilon|+\frac{1}{c_{1}^{1/2}}\log\left(
\delta_{\eta}c_{1}^{1/2}+(\delta_{\eta}^{2}c_{1}+\varepsilon)^{1/2}\right) \\
&  \leq\int_{a}^{a+\delta_{\eta}}\frac{1}{(\varepsilon+W(s))^{1/2}}\,ds\\
&  \leq\frac{1}{2c_{2}^{1/2}}|\log\varepsilon|+\frac{1}{c_{2}^{1/2}}%
\log\left(  \delta_{\eta}c_{0}^{1/2}+(\delta_{\eta}^{2}c_{0}+\varepsilon
)^{1/2}\right)
\end{align*}
Dividing by $|\log\varepsilon|$ and letting $\varepsilon\rightarrow0^{+}$, we
get%
\[
\frac{1}{2c_{1}^{1/2}}\leq\liminf_{\varepsilon\rightarrow0^{+}}\frac{\int%
_{a}^{a+\delta_{\eta}}\frac{1}{(\varepsilon+W(s))^{1/2}}\,ds}{|\log
\varepsilon|}\leq\limsup_{\varepsilon\rightarrow0^{+}}\frac{\int_{a}%
^{a+\delta_{\eta}}\frac{1}{(\varepsilon+W(s))^{1/2}}\,ds}{|\log\varepsilon
|}\leq\frac{1}{2c_{2}^{1/2}}.
\]
In turn, by (\ref{1d 203}),%
\[
\frac{1}{2c_{1}^{1/2}}\leq\liminf_{\varepsilon\rightarrow0^{+}}\frac{\int%
_{a}^{\beta}\frac{1}{(\varepsilon+W(s))^{1/2}}\,ds}{|\log\varepsilon|}%
\leq\limsup_{\varepsilon\rightarrow0^{+}}\frac{\int_{a}^{\beta}\frac
{1}{(\varepsilon+W(s))^{1/2}}\,ds}{|\log\varepsilon|}\leq\frac{1}{2c_{2}%
^{1/2}}.
\]
Letting $\eta\rightarrow0^{+}$ gives
\[
\lim_{\varepsilon\rightarrow0^{+}}\frac{\int_{a}^{\beta}\frac{1}%
{(\varepsilon+W(s))^{1/2}}\,ds}{|\log\varepsilon|}=\frac{1}{2^{1/2}%
(W^{\prime\prime}(a))^{1/2}}.
\]
Similarly, one can show that for every $a<\alpha<b$,%
\[
\lim_{\varepsilon\rightarrow0^{+}}\frac{\int_{\alpha}^{b}\frac{1}%
{(\varepsilon+W(s))^{1/2}}\,ds}{|\log\varepsilon|}=\frac{1}{2^{1/2}%
(W^{\prime\prime}(b))^{1/2}}.
\]
Hence,%
\begin{align*}
\frac{\int_{a}^{b}\frac{1}{(\varepsilon+W(s))^{1/2}}\,ds}{|\log\varepsilon|}
&  =\frac{\int_{c}^{b}\frac{1}{(\varepsilon+W(s))^{1/2}}\,ds}{|\log
\varepsilon|}+\frac{\int_{a}^{c}\frac{1}{(\varepsilon+W(s))^{1/2}}\,ds}%
{|\log\varepsilon|}\\
&  \rightarrow\frac{1}{2^{1/2}(W^{\prime\prime}(a))^{1/2}}+\frac{1}%
{2^{1/2}(W^{\prime\prime}(b))^{1/2}}.
\end{align*}
The inequality (\ref{near b log}) now follows. \hfill
\end{proof}

For the proof of the following proposition, we refer to
\cite[Proposition 2.3]{fonseca-kreutz-leoni2025I}.

\begin{proposition}
\label{proposition difference}Let $a\leq\alpha_{\varepsilon}\leq
\beta_{\varepsilon}\leq b$. Then, there exists a constant $C>0$ such that
\begin{equation}
\int_{\alpha_{\varepsilon}}^{\beta_{\varepsilon}}\left[  \frac{2}%
{(\delta+W(s))^{1/2}+W^{1/2}(s)}-\frac{1}{(\delta+W(s))^{1/2}}\right]
\,ds\leq C \label{2d 40}%
\end{equation}
for all $0<\delta<1$.
\end{proposition}

We assume that $g_{\varepsilon}:\partial\Omega\rightarrow\mathbb{R}$ and
$g:\partial\Omega\rightarrow\mathbb{R}$ satisfy the following hypotheses:
\begin{align}
g_{\varepsilon}  &  \in H^{1}(\partial\Omega),\label{g epsilon smooth a}\\
\varepsilon\int_{\partial\Omega}|\nabla_{\tau}g_{\varepsilon}|^{2}%
d\mathcal{H}^{N-1}  &  =o(1)\quad\text{as }\varepsilon\rightarrow
0^{+},\label{g epsilon bound derivatives a}\\
|g_{\varepsilon}(x)-g(x)|  &  \leq C\varepsilon^{\gamma},\quad x\in
\partial\Omega,\quad\gamma>1 \label{g epsilon -g bound a}%
\end{align}
for all $\varepsilon\in(0,1)$ and for some constant $C>0$. Here, $\nabla
_{\tau}$ denotes the tangential gradient.

\bigskip

In what follows, given $z\in\mathbb{R}^{N}$, with a slight abuse of notation,
we write%
\begin{equation}
z=(z^{\prime},z_{N})\in\mathbb{R}^{N-1}\times\mathbb{R}, \label{z' notation}%
\end{equation}
where $z^{\prime}:=(z_{1},\ldots,z_{N-1})$. We also write
\begin{equation}
\nabla^{\prime}:=\left(  \frac{\partial}{\partial z_{1}},\ldots,\frac
{\partial}{\partial z_{N-1}}\right)  . \label{grad' notation}%
\end{equation}
In what follows, given $\delta>0$ we define%
\begin{equation}
\Omega_{\delta}:=\{x\in\Omega:\,\operatorname*{dist}(x,\partial\Omega
)<\delta\}. \label{Omega delta}%
\end{equation}
For the proof of the following lemma, we refer to 
\cite[Lemma 2.6]{fonseca-kreutz-leoni2025I}. 

\begin{lemma}
\label{lemma diffeomorphism}Assume that $\Omega\subset\mathbb{R}^{N}$ is an
open, bounded, connected set and that its boundary $\partial\Omega$ is of
class $C^{2,d}$, $0<d\leq1$. If $\delta>0$ is sufficiently small, then  the mapping
\[
\Phi:\partial\Omega\times\lbrack0,\delta]\rightarrow\overline{\Omega}_{\delta}%
\]
given by%
\[
\Phi(y,t):= y+t\nu(y),
\]
where $\nu(y)$ is the unit inward normal vector to $\partial\Omega$ at $y$ and
$\Omega_{\delta}$ is defined in \eqref{Omega delta}, is a diffeomorphism of
class $C^{1,d}$. Moreover, $\Omega\setminus\Omega_{\delta}$ is connected for
all $\delta>0$ sufficiently small. Finally,%
\begin{equation}
\det J_{\Phi}(y,0)=1\quad\text{for all }y\in\partial\Omega\label{det=1}%
\end{equation}
and%
\begin{equation}
\frac{\partial}{\partial t}\left.  \det J_{\Phi}(y,t)\right\vert _{t=0}%
=\kappa(y)\quad\text{for all }y\in\partial\Omega, \label{curvature}%
\end{equation}
where $\kappa(y)$ is the mean curvature of $\partial\Omega$ at $y$.
\end{lemma}

\section{A 1D Functional Problem}

\label{section 1d functional}Let
\[
I:=(0,T)
\]
for some $T>0$ and consider a weight function%
\begin{equation}
\omega\in C^{1,d}([0,T]),\quad\min_{\lbrack0,T]}\omega>0. \label{etaSmooth}%
\end{equation}
The prototype we have in mind is given by
\[
\omega(t):= 1+t\kappa(t).
\]
In this section, we study the second-order $\Gamma$-convergence of the family of
functionals
\[
G_{\varepsilon}(v):=\int_{I}(W(v(t))+\varepsilon^{2}(v^{\prime}(t))^{2}%
)\omega(t)\,dt,\quad v\in\,H^{1}(I),
\]
subject to the Dirichlet boundary condition
\begin{equation}
v(0)=\alpha_{\varepsilon},\quad v(T)=\beta_{\varepsilon}. \label{1d dirichlet}%
\end{equation}

In what follows, we will need the weighted BV space $BV_{\omega}(I)$ given by
all functions $v\in BV_{\operatorname*{loc}}(I)$ for which the norm
\[
\Vert v\Vert_{BV_{\omega}}:=\int_{I}|v(t)|\omega(t)\,dt+\int_{I}%
\omega(t)\,d|Dv|(t)
\]
is finite. For $v\in BV_{\omega}(I)$ we will also write the weighted total
variation of the derivative in the following manner
\[
|Dv|_{\omega}(E):=\int_{E}\omega(t)\,d|Dv|(t).
\]
For a more detailed introduction to weighted BV spaces and their applications to phase field models, we refer to \cite{baldi2001,fonseca-liu2017}.

We will study the second-order $\Gamma$-convergence with respect to
the metric in $L^{1}(I)$. This choice is motivated by the following
compactness result.

\begin{theorem}
[Compactness]\label{theorem 1d compactness}Assume that $W$ satisfies \eqref{W_Smooth}-\eqref{W' three zeroes}, that $\omega$
satisfies \eqref{etaSmooth}, and that $\alpha_{\varepsilon
}\rightarrow\alpha$ and $\beta_{\varepsilon}\rightarrow\beta$ as
$\varepsilon\rightarrow0^{+}$ for some $\alpha,\beta\in\mathbb{R}$. Let
$\varepsilon_{n}\rightarrow0^{+}$ and $v_{n}\in$\thinspace$H^{1}(I)$ be such
that%
\[
\sup_{n}\int_{I}\left(  \frac{1}{\varepsilon_{n}}W(v_{n}(t))+\varepsilon
_{n}(v_{n}^{\prime}(t))^{2}\right)  \omega(t)\,dt<\infty.
\]
Then there exist a subsequence $\{v_{n_{k}}\}_{k}$ of $\{v_{n}\}_{n}$ and
$v\in BV_{\omega}(I;\{a,b\})$ such that $v_{n_{k}}\rightarrow v$ in
\thinspace$L^{1}(I)$.
\end{theorem}

The proof is identical to the one of \cite[Proposition 4.3]{leoni-murray2016}
and so we omit it. In view of the previous theorem, we extend $G_{\varepsilon
}$ to \thinspace$L^{1}(I)$ by setting
\begin{equation}
G_{\varepsilon}(v):=\left\{
\begin{array}
[c]{ll}%
\int_{I}(W(v(t))+\varepsilon^{2}(v^{\prime}(t))^{2})\omega(t)\,dt & \text{if
}v\in H^{1}(I)\text{ satisfies \eqref{1d dirichlet} }\\
\infty & \text{otherwise in }L^{1}(I).
\end{array}
\right.  \label{1d functional}%
\end{equation}

\subsection{Zeroth and First-Order $\Gamma$-limit of $G_{\varepsilon}$}

For the proof of the results in this subsection, we refer to
\cite{fonseca-kreutz-leoni2025I}. We begin by establishing the zeroth order
$\Gamma$-limit of the functional $G_{\varepsilon}$.

\begin{theorem}
\label{theorem 1d zero gamma}Assume that $W$ satisfies
\eqref{W_Smooth}-\eqref{W' three zeroes}, that $\omega$ satisfies \eqref{etaSmooth}, and that $\alpha_{\varepsilon}\rightarrow\alpha$
and $\beta_{\varepsilon}\rightarrow\beta$ as $\varepsilon\rightarrow0^{+}$ for
some $\alpha,\beta\in\mathbb{R}$. Then the family $\{G_{\varepsilon
}\}_{\varepsilon}$ $\Gamma$-converges to $G^{(0)}$ in \thinspace$L^{1}(I)$ as
$\varepsilon\rightarrow0^{+}$, where
\[
G^{(0)}(v):=\int_{I}W(v(t))\omega(t)\,dt.
\]

\end{theorem}

Since $W^{-1}(\{0\})=\{a,b\}$, it follows that
\[
\inf_{v\in\,L^{1}(I)}G^{(0)}(v)=0.
\]
Therefore,%
\begin{align}
G_{\varepsilon}^{(1)}(v)  &  :=\frac{G_{\varepsilon}(v)-\inf_{\,L^{1}%
(I)}G^{(0)}}{\varepsilon}\label{G1 epsilon}\\
&  =\int_{I}\left(  \frac{1}{\varepsilon}W(v(t))+\varepsilon(v^{\prime
}(t))^{2}\right)  \omega(t)\,dt\nonumber
\end{align}
if $v\in$\thinspace$H^{1}(I)$ satisfies \eqref{1d dirichlet} and
$G_{\varepsilon}^{(1)}(v):=\infty$ if $v\in$\thinspace$L^{1}(I)\backslash
$\thinspace$H^{1}(I)$ or if the boundary condition \eqref{1d dirichlet} fails.

We now characterize the first-order Gamma limit of the family
$\{G_{\varepsilon}\}_{\varepsilon}$.

\begin{theorem}
\label{theorem 1d first gamma}Assume that $W$ satisfies hypotheses
\eqref{W_Smooth}-\eqref{W' three zeroes}, that $\omega$ satisfies
hypothesis \eqref{etaSmooth}, and that $\alpha_{\varepsilon}\rightarrow\alpha$
and $\beta_{\varepsilon}\rightarrow\beta$ as $\varepsilon\rightarrow0^{+}$ for
some $\alpha,\beta\in\mathbb{R}$. Then the family $\{G_{\varepsilon}%
^{(1)}\}_{\varepsilon}$ $\Gamma$-converges to $G^{(1)}$ in \thinspace
$L^{1}(I)$ as $\varepsilon\rightarrow0^{+}$, where%
\[
G^{(1)}(v):=\left\{
\begin{array}
[c]{ll}%
\begin{split}\frac{C_{W}}{b-a}|Dv|_{\omega}(I)&+\operatorname*{d}\nolimits_{W}%
(v(0),\alpha)\omega(0)\\&+\operatorname*{d}\nolimits_{W}(v(T),\beta)\omega(T)\end{split} &
\text{if }v\in BV_{\omega}(I;\{a,b\}),\\
\infty & \text{otherwise in }\,L^{1}(I),
\end{array}
\right.
\]
where $\operatorname*{d}\nolimits_{W}$ and $C_{W}$ are defined in
\eqref{distance definition} and \eqref{cW definition}, respectively.
\end{theorem}

Next we show that if $\omega$ is sufficiently close to $\omega(0)$ or strictly
increasing, then the unique minimizer of $G^{(1)}$ is the constant function
$b$.

\begin{corollary}
\label{corollary 1d minimizer}Assume that $W$ satisfies
\eqref{W_Smooth}-\eqref{W' three zeroes} and let $a<\alpha<b$
and\ $\beta=b$. Suppose that $\omega$ satisfies \eqref{etaSmooth}
and that
\begin{equation}
\omega(t)>\omega(0)-\omega_{0}\quad\text{for all }t\in(0,T],
\label{eta close to eta0}%
\end{equation}
where
\begin{equation}
0\leq\omega_{0}<\frac{1}{2}\frac{C_{W}-\operatorname*{d}\nolimits_{W}%
(\alpha,b)}{C_{W}}\omega(0) \label{eta0}%
\end{equation}
if $a<\alpha$, while $\omega$ is strictly increasing if $\alpha=a$. Then the
unique minimizer of $G^{(1)}$ is the constant function $b$, with%
\[
\min_{L_{\omega}^{1}(I)}G^{(1)}(v)=G^{(1)}(b)=\operatorname*{d}\nolimits_{W}%
(\alpha,b)\omega(0).
\]

\end{corollary}

\begin{proof}
\textbf{Step 1:} Assume that $a<\alpha<b$. Let $v\in BV_{\omega}(I;\{a,b\})$.
If $v$ has at least one jump point at $t_{0}\in I$, then by
(\ref{eta close to eta0}) and (\ref{eta0}),%
\[
G^{(1)}(v)\geq\frac{C_{W}}{b-a}|Dv|_{\omega}(I)\geq C_{W}\omega(t_{0}%
)>C_{W}(\omega(0)-\omega_{0})\geq\operatorname*{d}\nolimits_{W}(\alpha
,b)\omega(0).
\]
Hence, either $v\equiv b$ or $v\equiv a$. If $v\equiv  a$, then again by
(\ref{eta close to eta0}) and (\ref{eta0})
\[
G^{(1)}(a)=\operatorname*{d}\nolimits_{W}(a,\alpha)\omega(0)+C_{W}%
\omega(T)>C_{W}(\omega(0)-\omega_{0})\geq\operatorname*{d}\nolimits_{W}%
(\alpha,b)\omega(0).
\]
\textbf{Step 2:} Assume that $\alpha=a$ and $\beta=b$. Let $v\in BV_{\omega
}(I;\{a,b\})$. If $v$ has at least one jump point at $t_{0}\in I$, then since
$\omega$ is strictly increasing
\[
G^{(1)}(v)\geq\frac{C_{W}}{b-a}|Dv|_{\omega}(I)\geq C_{W}\omega(t_{0}%
)>C_{W}\omega(0).
\]
Hence, either $v\equiv b$ or $v\equiv a$. If $v \equiv a$, then again by
(\ref{eta close to eta0}) and (\ref{eta0})
\[
G^{(1)}(a)=C_{W}\omega(T)>C_{W}\omega(0).
\]
This completes the proof. \hfill
\end{proof}

\begin{remark}
Note that condition \eqref{eta close to eta0} holds if either $\omega$ is
strictly increasing, with $\omega_{0}=0$, or if $T$ is sufficiently small, by
continuity of $\omega$.
\end{remark}

\subsection{Second-Order $\Gamma$-limsup}

The scaling of the second-order asymptotic development via $\Gamma$-convergence
of $G_{\varepsilon}$ changes depending on whether $a<\alpha$ and $a=a$. When
$a<\alpha$, under the hypotheses of Corollary \ref{corollary 1d minimizer}, we
have
\[
\min_{\,L^{1}(I)}G^{(1)}(v)=G^{(1)}(b)=\operatorname*{d}\nolimits_{W}%
(\alpha,b)\omega(0).
\]
In this case, we define%
\begin{align}
G_{\varepsilon}^{(2)}(v)  &  :=\frac{G_{\varepsilon}^{(1)}(v)-\inf
_{\,L^{1}(I)}G^{(1)}}{\varepsilon}\label{G 2 epsilon a<alpha}\\
&  =\int_{I}\left(  \frac{1}{\varepsilon^{2}}W(v(t))+(v^{\prime}%
(t))^{2}\right)  \omega(t)\,dt-\operatorname*{d}\nolimits_{W}(\alpha
,b)\omega(0)\frac{1}{\varepsilon}\nonumber
\end{align}
if $v\in$\thinspace$H^{1}(I)$ satisfies \eqref{1d dirichlet} and
$G_{\varepsilon}^{(2)}(v):=\infty$ if $v\in$\thinspace$L^{1}(I)\backslash
$\thinspace$H^{1}(I)$ or if the boundary condition \eqref{1d dirichlet} fails.
For the proof of the following theorem, we refer to
\cite{fonseca-kreutz-leoni2025I}.

\begin{theorem}
[Second-Order Limsup, $a<\alpha$]\label{theorem 1d limsup a<alpha}Assume that
$W$ satisfies \eqref{W_Smooth}-\eqref{W' three zeroes}, that
$\alpha_{-}$ satisfies \eqref{alpha and beta minus}, and that
$\omega$ satisfies \eqref{etaSmooth}, \eqref{eta close to eta0},
where%
\begin{equation}
0\leq\omega_{0}<\frac{1}{2}\frac{\operatorname*{d}\nolimits_{W}(a,\alpha_{-}%
)}{C_{W}}\omega(0). \label{eta 0 alpha-}%
\end{equation}
Let
\[
\alpha_{-}\leq\alpha_{\varepsilon},\,\beta_{\varepsilon}\leq b,
\]
with
\begin{equation}
|\alpha_{\varepsilon}-\alpha|\leq A_{0}\varepsilon^{\gamma},\quad
|\beta_{\varepsilon}-b|\leq B_{0}\varepsilon^{\gamma}
\label{alpha epsilon and beta epsilon a<alpha}%
\end{equation}
for some $\alpha$,$\,\beta$ and where $A_{0}$, $B_{0}>0$, and $\gamma>1$. Then
there exist constants $0<\varepsilon_{0}<1$, $C,C_{0}>0$, and $\gamma
_{0},\gamma_{1}>0$, depending only on $\alpha_{-}$, $A_{0}$, $B_{0}$, $T$,
$\omega$, and $W$, and functions $v_{\varepsilon}\in$\thinspace$H^{1}(I)$
satisfying \eqref{1d dirichlet}, $a\leq v_{\varepsilon}\leq b$, and
$v_{\varepsilon}\rightarrow b$ in \thinspace$L^{1}(I)$, such that
\begin{equation}
G_{\varepsilon}^{(2)}(v_{\varepsilon})\leq\int_{0}^{l}2W(p_{\varepsilon
}(t))t\,dt\,\omega^{\prime}(0)+Ce^{-2\sigma l}\left(  2\sigma l+1\right)
+C\varepsilon^{2\gamma}l+C\varepsilon^{\gamma_{1}}|\log\varepsilon
|^{\gamma_{0}} \label{1d 000}%
\end{equation}
for all $0<\varepsilon<\varepsilon_{0}$ and all $l>0$, where $p_{\varepsilon
}(t):= v_{\varepsilon}(\varepsilon t)$ is such that $p_{\varepsilon}\rightarrow
z_{\alpha}$ pointwise in $[0,\infty)$, where $z_{\alpha}$ solves the Cauchy
problem \eqref{cauchy problem z alpha} and $G_{\varepsilon}^{(2)}$ is defined
in \eqref{G 2 epsilon a<alpha}. In particular,%
\[
\limsup_{\varepsilon\rightarrow0^{+}}G_{\varepsilon}^{(2)}(v_{\varepsilon
})\leq\int_{0}^{\infty}2W^{1/2}(z_{\alpha}(t))z_{\alpha}^{\prime
}(t)t\,dt\,\omega^{\prime}(0).
\]

\end{theorem}

\begin{remark}
\label{remark limsup a<alpha}The function $v_{\varepsilon}$ is constructed as
the inverse function of the function
\[
\Psi_{\varepsilon}(r):=\int_{\alpha_{\varepsilon}}^{r}\frac{\varepsilon
}{(\delta_{\varepsilon}+W(s))^{1/2}}\,ds,
\]
where $\delta_{\varepsilon}\rightarrow0^{+}$ goes to zero faster than
$\varepsilon$. Observe that if we take $\delta_{\varepsilon}=\varepsilon$,
then \eqref{1d 000} should be replaced by%
\begin{align*}
G_{\varepsilon}^{(2)}(v_{\varepsilon})  &  \leq C+\int_{0}^{l}%
2W(p_{\varepsilon}(s))s\,ds\omega^{\prime}(0)+Ce^{-2\sigma l}\left(  2\sigma
l+1\right) \\
&  \quad+C\varepsilon^{2\gamma}l+C\varepsilon\log^{2}\varepsilon
+C\varepsilon^{d}|\log\varepsilon|^{1+d}+C\varepsilon^{2\gamma-2}.
\end{align*}

\end{remark}

On the other hand, when $\alpha=a$, again under the hypotheses of Corollary
\ref{corollary 1d minimizer}, we have
\[
\min_{\,L^{1}(I)}G^{(1)}(v)=G^{(1)}(b)=C_{W}\omega(0).
\]
In this case, we define%
\begin{align}
G_{\varepsilon}^{(2)}(v)  &  :=\frac{G_{\varepsilon}^{(1)}(v)-\inf
_{\,L^{1}(I)}G^{(1)}}{\varepsilon|\log\varepsilon|}\label{G 2 epsilon a=alpha}%
\\
&  =\frac{1}{\varepsilon|\log\varepsilon|}\int_{I}\left(  \frac{1}%
{\varepsilon}W(v(t))+\varepsilon(v^{\prime}(t))^{2}\right)  \omega
(t)\,dt-C_{W}\omega(0)\frac{1}{\varepsilon|\log\varepsilon|}\nonumber
\end{align}
if $v\in$\thinspace$H^{1}(I)$ satisfies \eqref{1d dirichlet} and
$G_{\varepsilon}^{(2)}(v):=\infty$ if $v\in$\thinspace$L^{1}(I)\backslash
$\thinspace$H^{1}(I)$ or if the boundary condition \eqref{1d dirichlet} fails.

We study the second-order $\Gamma$-limsup of the family $\{G_{\varepsilon
}\}_{\varepsilon}$.

\begin{theorem}
[Second-Order Limsup, $\alpha=a$]\label{theorem 1d limsup alpha=a}Assume that
$W$ satisfies \eqref{W_Smooth}-\eqref{W' three zeroes}, that
$\alpha_{-},\beta_{-}$ satisfy \eqref{alpha and beta minus}, and
that $\omega$ satisfies \eqref{etaSmooth} and is strictly
increasing with $\omega^{\prime}(0)>0$. Let $\alpha_{-}\leq\alpha
_{\varepsilon}<\beta_{-}\leq\beta_{\varepsilon}<b$ with
\begin{equation}
|\alpha_{\varepsilon}-a|\leq A_{0}\varepsilon^{\gamma},\quad|\beta
_{\varepsilon}-b|\leq B_{0}\varepsilon^{\gamma},
\label{alpha epsilon and beta epsilon}%
\end{equation}
where $A_{0}$, $B_{0}>0$, and $\gamma>1$. There exist $v_{\varepsilon}\in
H_{\omega}^{1}(I)$ satisfying \eqref{1d dirichlet}, such that $v_{\varepsilon
}\rightarrow b$ in $L_{\omega}^{1}(I)$ and, for every $0<\eta<1$,%
\begin{equation}
G_{\varepsilon}^{(2)}(v_{\varepsilon})\leq(1+\eta)\frac{C_{W}}{2^{1/2}%
(W^{\prime\prime}(a))^{1/2}}\omega^{\prime}(0)+\frac{C}{|\log\varepsilon|}
\label{1d limsup alpha=a}%
\end{equation}
for all $0<\varepsilon<\varepsilon_{\eta}$, for some $0<\varepsilon_{\eta}<1$
depending on $\eta$, $A_{0}$, $B_{0}$, $T$, $\omega$, and $W$, and for some
constant $C>0$, depending on $A_{0}$, $B_{0}$, $T$, $\omega$, and $W$, and
where $G_{\varepsilon}^{(2)}$ is defined in \eqref{G 2 epsilon a=alpha}. In
particular,%
\begin{equation}
\limsup_{\varepsilon\rightarrow0^{+}}G_{\varepsilon}^{(2)}(v_{\varepsilon
})\leq\frac{C_{W}}{2^{1/2}(W^{\prime\prime}(a))^{1/2}}\omega^{\prime}(0).
\label{1d limsup alpha=a no eta}%
\end{equation}

\end{theorem}

\begin{proof}
In this proof, $\varepsilon_{0}$ and $C$ depend only on $A_{0}$, $B_{0}$, $T$, $\omega$, $W$. In what follows, we will
take $\varepsilon_{0}$ smaller and $C$ larger, if necessary, preserving the
same dependence on the parameters.

Define
\[
\Psi_{\varepsilon}(r):=\int_{\alpha_{\varepsilon}}^{r}\frac{\varepsilon
}{(\varepsilon+W(s))^{1/2}}\,ds.
\]
Let
\begin{equation}
0\leq L_{\varepsilon}:=\Psi_{\varepsilon}(c)<T_{\varepsilon}:=\Psi
_{\varepsilon}(\beta_{\varepsilon}). \label{1d xi epsilon a}%
\end{equation}
By (\ref{near b log}) and the fact that $a\leq\alpha_{\varepsilon}%
,\beta_{\varepsilon}\leq b$, we have
\begin{equation}
L_{\varepsilon}\leq T_{\varepsilon}\leq\int_{a}^{b}\frac{\varepsilon
}{(\varepsilon+W(s))^{1/2}}\,ds\leq C\varepsilon|\log\varepsilon|
\label{1d L epsilon}%
\end{equation}
for all $0<\varepsilon<\varepsilon_{0}$.

Let $v_{\varepsilon}:[0,T_{\varepsilon}]\rightarrow\lbrack\alpha_{\varepsilon
},\beta_{\varepsilon}]$ be the inverse of $\Psi_{\varepsilon}$. Then
$v_{\varepsilon}\left(  0\right)  =\alpha_{\varepsilon}$, $v_{\varepsilon
}(T_{\varepsilon})=\beta_{\varepsilon}$, and
\begin{equation}
v_{\varepsilon}^{\prime}(t)=\frac{(\varepsilon+W(v_{\varepsilon}\left(
t\right)  ))^{1/2}}{\varepsilon}. \label{1d 556a}%
\end{equation}
Extend $v_{\varepsilon}$ to be equal to $\beta_{\varepsilon}$ for
$t>T_{\varepsilon}$.

Since $\omega\in C^{1,d}(I)$, by Taylor's formula, for $t\in\lbrack0,T]$,%
\[
\omega(t)=\omega(0)+\omega^{\prime}(0)t+R_{1}(t),
\]
where%
\begin{equation}
|R_{1}(t)|=|\omega^{\prime}(\theta t)-\omega^{\prime}(0)|t\leq|\omega^{\prime
}|_{C^{0,d}}t^{1+d}. \label{1d 103aa}%
\end{equation}
Write%
\begin{align}
G_{\varepsilon}^{(2)}(v_{\varepsilon})  &  =\left[  \int_{0}^{T_{\varepsilon}%
}\left(  \frac{1}{\varepsilon}W(v_{\varepsilon})+\varepsilon(v_{\varepsilon
}^{\prime})^{2}\right)  \,dt-C_{W}\right]  \frac{\omega(0)}{\varepsilon
|\log\varepsilon|}\nonumber\\
&  \quad+\int_{0}^{T_{\varepsilon}}\left(  \frac{1}{\varepsilon}%
W(v_{\varepsilon})+\varepsilon(v_{\varepsilon}^{\prime})^{2}\right)
t\,dt\frac{\omega^{\prime}(0)}{\varepsilon|\log\varepsilon|}\label{1d 557a}\\
&  \quad+\int_{0}^{T_{\varepsilon}}\left(  \frac{1}{\varepsilon}%
W(v_{\varepsilon})+\varepsilon(v_{\varepsilon}^{\prime})^{2}\right)
R_{1}\,dt\frac{1}{\varepsilon|\log\varepsilon|}\nonumber\\
&  \quad+\int_{T_{\varepsilon}}^{T}\left(  \frac{1}{\varepsilon}%
W(v_{\varepsilon})+\varepsilon(v_{\varepsilon}^{\prime})^{2}\right)
\omega\,dt\frac{1}{\varepsilon|\log\varepsilon|}=:\mathcal{A}+\mathcal{B}%
+\mathcal{C}+\mathcal{D}.\nonumber
\end{align}
\textbf{Step 1. }We estimate $\mathcal{A}$. By (\ref{1d 556a}), the change of
variables $s=v_{\varepsilon}(t)$, and the equality
\[
(A+B)^{1/2}-B^{1/2}=\frac{A}{(A+B)^{1/2}+B^{1/2}},
\]
we have%
\begin{align*}
\int_{0}^{T_{\varepsilon}}  &  \left(  \frac{1}{\varepsilon}W(v_{\varepsilon
})+\varepsilon(v_{\varepsilon}^{\prime})^{2}\right)  \,dt=\int_{0}%
^{T_{\varepsilon}}\left(  \frac{1}{\varepsilon}(\varepsilon+W(v_{\varepsilon
}))+\varepsilon(v_{\varepsilon}^{\prime})^{2}\right)  \,dt-T_{\varepsilon}\\
&  =\int_{0}^{T_{\varepsilon}}2(\varepsilon+W(v_{\varepsilon}))^{1/2}%
v_{\varepsilon}^{\prime}\,dt-T_{\varepsilon}\\
&  =\int_{\alpha_{\varepsilon}}^{\beta_{\varepsilon}}2(\varepsilon
+W(s))^{1/2}\,ds-\int_{\alpha_{\varepsilon}}^{\beta_{\varepsilon}}%
\frac{\varepsilon}{(\varepsilon+W(s))^{1/2}}\,ds\\
&  =\int_{\alpha_{\varepsilon}}^{\beta_{\varepsilon}}2W^{1/2}(s)\,ds+\int%
_{\alpha_{\varepsilon}}^{\beta_{\varepsilon}}\left[  \frac{2\varepsilon
}{(\varepsilon+W(s))^{1/2}+W^{1/2}(s)}-\frac{\varepsilon}{(\varepsilon
+W(s))^{1/2}}\right]  \,ds.
\end{align*}
By Proposition \ref{proposition difference},%
\[
\int_{\alpha_{\varepsilon}}^{\beta_{\varepsilon}}\left[  \frac{2\varepsilon
}{(\varepsilon+W(s))^{1/2}+W^{1/2}(s)}-\frac{\varepsilon}{(\varepsilon
+W(s))^{1/2}}\right]  \,ds\leq C\varepsilon
\]
for all $0<\varepsilon<\varepsilon_{0}$. Hence, using also the fact that
$a<\alpha_{\varepsilon}<\beta_{\varepsilon}<b$, we obtain%
\begin{equation}
\int_{0}^{T_{\varepsilon}}\left(  \frac{1}{\varepsilon}W(v_{\varepsilon
})+\varepsilon(v_{\varepsilon}^{\prime})^{2}\right)  \,dt\leq C_{W}%
+C\varepsilon, \label{1d 200a}%
\end{equation}
and so%
\[
\mathcal{A}\leq C\frac{1}{|\log\varepsilon|}%
\]
for all $0<\varepsilon<\varepsilon_{0}$.

\textbf{Step 2. }We estimate $\mathcal{B}$ in (\ref{1d 557a}). By
(\ref{1d 556a}) and the change of variables $t:= r+L_{\varepsilon}$,%
\begin{align*}
\mathcal{B}  &  =\int_{-L_{\varepsilon}}^{T_{\varepsilon}-L_{\varepsilon}%
}\left(  \frac{1}{\varepsilon}W(\bar{v}_{\varepsilon})+\varepsilon(\bar
{v}_{\varepsilon}^{\prime})^{2}\right)  \,dr\frac{\omega^{\prime
}(0)L_{\varepsilon}}{\varepsilon|\log\varepsilon|}+\int_{-L_{\varepsilon}%
}^{T_{\varepsilon}-L_{\varepsilon}}\left(  \frac{1}{\varepsilon}W(\bar
{v}_{\varepsilon})+\varepsilon(\bar{v}_{\varepsilon}^{\prime})^{2}\right)
r\,dr\frac{\omega^{\prime}(0)}{\varepsilon|\log\varepsilon|}\\
&  =:\mathcal{B}_{1}+\mathcal{B}_{2},
\end{align*}
where $\bar{v}_{\varepsilon}(r):=v_{\varepsilon}(r+L_{\varepsilon})$. By
(\ref{1d 200a}), (\ref{1d L epsilon}) and the fact that $\omega^{\prime}%
(0)>0$,%
\begin{align*}
\mathcal{B}_{1}  &  \leq C_{W}\frac{\omega^{\prime}(0)L_{\varepsilon}%
}{\varepsilon|\log\varepsilon|}+C\frac{L_{\varepsilon}}{|\log\varepsilon|}\\
&  \leq C_{W}\frac{\omega^{\prime}(0)}{|\log\varepsilon|}\int_{a}^{c}\frac
{1}{(\varepsilon+W(\rho))^{1/2}}\,d\rho+C\varepsilon
\end{align*}
for all $0<\varepsilon<\varepsilon_{0}$. By \ref{limit integral}, given
$0<\eta<1$, there exists $0<\varepsilon_{\eta}<1$ such that%
\[
C_{W}\frac{\omega^{\prime}(0)}{|\log\varepsilon|}\int_{a}^{a+\eta}\frac
{1}{(\varepsilon+W(\rho))^{1/2}}\,d\rho\leq(1+\eta)\frac{C_{W}\omega^{\prime
}(0)}{2^{1/2}(W^{\prime\prime}(a))^{1/2}}%
\]
for all $0<\varepsilon<\varepsilon_{\eta}$.

On the other hand, by the change of variables $r:=\varepsilon s$,
\begin{align*}
\mathcal{B}_{2}  &  =\int_{-L_{\varepsilon}}^{T_{\varepsilon}-L_{\varepsilon}%
}2W(\bar{v}_{\varepsilon})r\,dr\frac{\omega^{\prime}(0)}{\varepsilon^{2}%
|\log\varepsilon|}+\int_{-L_{\varepsilon}}^{T_{\varepsilon}-L_{\varepsilon}%
}r\,dr\frac{\varepsilon\omega^{\prime}(0)}{\varepsilon^{2}|\log\varepsilon|}\\
&  =\int_{-L_{\varepsilon}\varepsilon^{-1}}^{(T_{\varepsilon}-L_{\varepsilon
})\varepsilon^{-1}}2W(p_{\varepsilon}(s))s\,ds\frac{\omega^{\prime}(0)}%
{|\log\varepsilon|}+\frac{\varepsilon\omega^{\prime}(0)[(T_{\varepsilon
}-L_{\varepsilon})^{2}-L_{\varepsilon}^{2}]}{2\varepsilon^{2}|\log
\varepsilon|}\\
&  :=\mathcal{B}_{2,1}+\mathcal{B}_{2,2},
\end{align*}
where $p_{\varepsilon}(s):=\bar{v}_{\varepsilon}(\varepsilon s)=v_{\varepsilon
}(\varepsilon s+L_{\varepsilon})$ solves the Cauchy problem%
\[
\left\{
\begin{array}
[c]{l}%
p_{\varepsilon}^{\prime}(s)=(\varepsilon+W(p_{\varepsilon}\left(  s\right)
))^{1/2},\\
p_{\varepsilon}(0)=c,
\end{array}
\right.
\]
in $[-L_{\varepsilon}\varepsilon^{-1},(T_{\varepsilon}-L_{\varepsilon
})\varepsilon^{-1}]$. Since $c\leq p_{\varepsilon}(s)\leq\beta_{\varepsilon
}<b$ for $0\leq s\leq(T_{\varepsilon}-L_{\varepsilon})\varepsilon^{-1}$, by
(\ref{W near b}) we have that%
\[
p_{\varepsilon}^{\prime}(s)\geq(W(p_{\varepsilon}\left(  s\right)
))^{1/2}\geq\sigma(b-p_{\varepsilon}(s))>0,
\]
and so
\[
-\sigma\geq\frac{(b-p_{\varepsilon}(s))^{\prime}}{b-p_{\varepsilon}(s)}%
=(\log(b-p_{\varepsilon}(s)))^{\prime}.
\]
Upon integration, we get%
\[
0\leq b-p_{\varepsilon}(s)\leq(b-c)e^{-\sigma s}\leq(b-c)e^{-\sigma s}.
\]
In turn, again by (\ref{W near b}), for $s\in\lbrack0,(T_{\varepsilon
}-L_{\varepsilon})\varepsilon^{-1}]$,%
\begin{equation}
W(p_{\varepsilon}(s))\leq\sigma^{-2}(b-p_{\varepsilon}(s))^{2}\leq\sigma
^{-2}(b-c)^{2}e^{-2\sigma s}. \label{1d 204}%
\end{equation}
On the other hand,  we claim that there exists $C>0$ such that 
\begin{align}
-C\int_{-L_{\varepsilon}\varepsilon^{-1}}^{0}e^{2\sigma s}|s|\,ds\leq \int_{-L_{\varepsilon}\varepsilon^{-1}}^{0}2W(p_{\varepsilon}(s))s\,ds\leq0. \label{ineqWs}
\end{align}
As $W\geq 0$ and $s \leq 0$ it is immediate that 
\[
\int_{-L_{\varepsilon}\varepsilon^{-1}}^{0}2W(p_{\varepsilon}(s))s\,ds\leq0.
\] 
Additionally, by (\ref{W near a}) for $-L_{\varepsilon}\varepsilon^{-1}\leq s\leq0$, we have
that%
\[
p_{\varepsilon}^{\prime}(s)\geq(W(p_{\varepsilon}\left(  s\right)
))^{1/2}\geq\sigma(p_{\varepsilon}(s)-a)>0,
\]
and so
\[
(\log(p_{\varepsilon}(s)-a))^{\prime}=\frac{(p_{\varepsilon}(s)-a)^{\prime}%
}{p_{\varepsilon}(s)-a}\geq\sigma.
\]
Upon integration, we get%
\[
\log\frac{c-a}{p_{\varepsilon}(s)-a}\geq\sigma(0-s)
\]
and so%
\[
c-a\geq(p_{\varepsilon}(s)-a)e^{-s\sigma},
\]
which gives%
\[
0\leq p_{\varepsilon}(s)-a\leq(c-a)e^{\sigma s}.
\]
In turn, again by (\ref{W near a}), for $s\in\lbrack-L_{\varepsilon
}\varepsilon^{-1},0]$,%
\[
W(p_{\varepsilon}(s))\leq\sigma^{2}(p_{\varepsilon}(s)-a)^{2}\leq\sigma
^{2}(c-a)^{2}e^{2\sigma s}.
\]
This implies \eqref{ineqWs} and therefore, using \eqref{1d 204}, we obtain
\[
\mathcal{B}_{2,1}\leq C\int_{0}^{\infty}e^{-2\sigma s}s\,ds\frac
{\omega^{\prime}(0)}{|\log\varepsilon|}\leq\frac{C}{|\log\varepsilon|}%
\] 
for all $0<\varepsilon<\varepsilon_{0}$. By (\ref{1d L epsilon}) and the fact that $\omega^{\prime}(0)>0$,
\[
\mathcal{B}_{2,2}\leq C\frac{\varepsilon T_{\varepsilon}^{2}}{\varepsilon
^{2}|\log\varepsilon|}\leq C\varepsilon^{2}|\log^{2}\varepsilon|
\]
for all $0<\varepsilon<\varepsilon_{0}$.

\textbf{Step 3. }We estimate $\mathcal{C}$ in (\ref{1d 557a}). Observe that by
(\ref{1d 200a}), \eqref{1d 103aa}, and (\ref{1d L epsilon}),
\begin{align*}
\mathcal{C}  &  \leq\int_{0}^{T_{\varepsilon}}\left(  \frac{1}{\varepsilon
}W(v_{\varepsilon})+\varepsilon(v_{\varepsilon}^{\prime})^{2}\right)
\,dt\frac{|\omega^{\prime}|_{C^{0,d}}T_{\varepsilon}^{1+d}}{\varepsilon
|\log\varepsilon|}\\
&  \leq C\varepsilon^{d}|\log\varepsilon|^{d}\left(  C_{W}+C\varepsilon
|\log\varepsilon|\right)  \leq C\varepsilon^{d}|\log\varepsilon|^{d}%
\end{align*}
for all $0<\varepsilon<\varepsilon_{0}$.

\textbf{Step 4. }We estimate $\mathcal{D}$ in (\ref{1d 557a}). By
(\ref{W near b}) and (\ref{alpha epsilon and beta epsilon}), for $t\geq
T_{\varepsilon}$,
\[
\mathcal{D}=W(\beta_{\varepsilon})\int_{T_{\varepsilon}}^{T}\omega\,dt\frac
{1}{\varepsilon^{2}|\log\varepsilon|}\leq\sigma^{-2}(b-\beta_{\varepsilon
})^{2}\int_{0}^{T}\omega\,dt\frac{1}{\varepsilon^{2}|\log\varepsilon|}\leq
C\frac{\varepsilon^{2\gamma-2}}{|\log\varepsilon|}%
\]
for all $0<\varepsilon<\varepsilon_{0}$.

Combining the estimates for $\mathcal{A}$, $\mathcal{B}$, $\mathcal{C}$, and
$\mathcal{D}$ gives (\ref{1d limsup alpha=a}). In turn, letting first
$\varepsilon\rightarrow0^{+}$ and then $\eta\rightarrow0^{+}$ in
(\ref{1d limsup alpha=a}) proves (\ref{1d limsup alpha=a no eta}).\hfill
\end{proof}

\subsection{Properties of Minimizers of $G_{\varepsilon}$}

In this subsection, we study some qualitative properties of the minimizers of the
functional $G_{\varepsilon}$ defined in \eqref{1d functional}:%
\begin{equation}
G_{\varepsilon}(v):=\int_{I}(W(v(t))+\varepsilon^{2}(v^{\prime}(t))^{2}%
)\omega(t)\,dt,\quad v\in\,H^{1}(I), \label{1d functional G}%
\end{equation}
subject to the Dirichlet boundary conditions
\begin{equation}
v_{\varepsilon}(0)=\alpha_{\varepsilon},\quad v_{\varepsilon}(T)=\beta
_{\varepsilon}. \label{1d dirichlet 1}%
\end{equation}

Theorem \ref{theorem 1d EL}, Corollary \ref{corollary bounded derivative}, Theorem \ref{theorem monotonicity}, and Theorem \ref{theorem barrier 1d} have been proven in \cite{fonseca-kreutz-leoni2025I}, cf. \cite[Theorem 3.8]{fonseca-kreutz-leoni2025I}, \cite[Corollary 3.9]{fonseca-kreutz-leoni2025I}, \cite[Theorem 3.12]{fonseca-kreutz-leoni2025I},  and \cite[Theorem 3.10]{fonseca-kreutz-leoni2025I}. We state them for the convenience of the reader.

\begin{theorem}
\label{theorem 1d EL}Assume that $W$ satisfies
\eqref{W_Smooth}-\eqref{W' three zeroes}, that $\omega$ satisfies \eqref{etaSmooth}, and that $a\leq\alpha_{\varepsilon},$
$\beta_{\varepsilon}\leq b$. Then the functional $G_{\varepsilon}$ admits a
minimizer $v_{\varepsilon}\in$\thinspace$H^{1}(I)$. Moreover, $v_{\varepsilon
}\in C^{2}([0,T])$, $v_{\varepsilon}$ satisfies the Euler--Lagrange equations
\begin{equation}
2\varepsilon^{2}(v_{\varepsilon}^{\prime}(t)\omega(t))^{\prime}-W^{\prime
}(v_{\varepsilon}(t))\omega(t)=0, \label{1d euler lagrange}%
\end{equation}
and $v_{\varepsilon}\equiv a$, or $v_{\varepsilon}\equiv b$, or
\begin{equation}
a<v_{\varepsilon}(t)<b\quad\text{for all }t\in(0,T). \label{1d truncated}%
\end{equation}

\end{theorem}

\begin{corollary}
\label{corollary bounded derivative}Assume that $W$ satisfies
\eqref{W_Smooth}-\eqref{W' three zeroes}, that $\omega$ satisfies \eqref{etaSmooth}, and that $a\leq\alpha_{\varepsilon},$
$\beta_{\varepsilon}\leq b$. Let $v_{\varepsilon}$ be the minimizer of
$G_{\varepsilon}$ obtained in Theorem \ref{theorem 1d EL}. Then there exists a
constant $C_{0}>0$, depending only on $\omega$, $T$, $a$, $b$, and $W$, such
that%
\[
|v_{\varepsilon}^{\prime}(t)|\leq\frac{C_{0}}{\varepsilon}\quad\text{for all
}t\in I
\]
and for every $0<\varepsilon<1$.
\end{corollary}

Next, we recall some differential inequalities for $v_{\varepsilon}$.

\begin{theorem}
\label{theorem monotonicity}Assume that $W$ satisfies
\eqref{W_Smooth}-\eqref{W' three zeroes}, that $\omega$ satisfies \eqref{etaSmooth}, and that $a\leq\alpha_{\varepsilon},$
$\beta_{\varepsilon}\leq b$. Let $v_{\varepsilon}$ be the minimizer of
$G_{\varepsilon}$ obtained in Theorem \ref{theorem 1d EL} and let $\alpha
_{-},\beta_{-}$ be given as in \eqref{alpha and beta minus}. Then there exists
a constant $C>0$ such that
\begin{equation}
\varepsilon(v_{\varepsilon}^{\prime}(0))^{2}-\frac{1}{\varepsilon}%
W(\alpha_{\varepsilon})\leq C \label{ineq:initial_value_veps}
\end{equation}
for all $0<\varepsilon<1$. Moreover, there exist a constant $\tau_{0}>0$,
depending only on $\omega$, $T$, $a$, $b$, $\alpha_{-}$, $\beta_{-}$ and $W$,
such that
\begin{equation}
\frac{1}{2}\sigma^{2}(v_{\varepsilon}(t)-a)^{2}\leq\varepsilon^{2}%
(v_{\varepsilon}^{\prime}(t))^{2}\leq\frac{3}{2}\sigma^{-2}(v_{\varepsilon
}(t)-a)^{2} \label{1d v' near a}%
\end{equation}
whenever $a+\tau_{0}\varepsilon^{1/2}\leq v_{\varepsilon}(t)\leq\beta_{-}$
and
\begin{equation}
\frac{1}{2}\sigma^{2}(b-v_{\varepsilon}(t))^{2}\leq\varepsilon^{2}%
(v_{\varepsilon}^{\prime}(t))^{2}\leq\frac{3}{2}\sigma^{-2}(b-v_{\varepsilon
}(t))^{2} \label{1d v' near b}%
\end{equation}
whenever $\alpha_{-}\leq v_{\varepsilon}(t)\leq b-\tau_{0}\varepsilon^{1/2}$,
where $\sigma>0$ is the constant given in Remark \ref{remark W near b}.
\end{theorem}

\begin{theorem}
\label{theorem barrier 1d}Assume that $W$ satisfies
\eqref{W_Smooth}-\eqref{W' three zeroes}, that $\omega$ satisfies \eqref{etaSmooth}, and that that $a\leq\alpha_{\varepsilon},$
$\beta_{\varepsilon}\leq b$. Let $v_{\varepsilon}$ be the minimizer of
$G_{\varepsilon}$ obtained in Theorem \ref{theorem 1d EL} and let
\begin{align}
A_{\varepsilon}  &  :=\{t\in\lbrack0,T]:~\alpha_{\varepsilon}+\varepsilon
^{k}\leq v_{\varepsilon}(t)\leq\alpha_{-}\},\label{1d set A epsilon}\\
B_{\varepsilon}  &  :=\{t\in\lbrack0,T]:~\beta_{-}\leq v_{\varepsilon}%
(t)\leq\beta_{\varepsilon}-\varepsilon^{k}\}. \label{1d set B epsilon}%
\end{align}
Then there exist $C>0$ and $0<\varepsilon_{0}<1$ depending only on $a_{-}$,
$\beta_{-}$, $T$, $\omega$, $W$, such that if $I_{\varepsilon}$ is a maximal
subinterval of $A_{\varepsilon}$ or $B_{\varepsilon}$, then
\begin{equation}
\operatorname*{diam}I_{\varepsilon}\leq C\varepsilon|\log\varepsilon|
\label{1d diam I epsilon}%
\end{equation}
for all $0<\varepsilon<\varepsilon_{0}$.
\end{theorem}

Next, we strengthen the hypotheses on the Dirichlet data $\alpha_{\varepsilon
}$ and $\beta_{\varepsilon}$ and derive additional properties of minimizers.

Given $0<\eta<\frac{1}{4}$, by Taylor's formula and the fact that
$W^{\prime\prime}(a)>0$, we can find $\delta_{\eta}>0$ such such that%
\begin{equation}
\frac{1}{2}W^{\prime\prime}(a)(1-\eta)(s-a)^{2}\leq W(s)\leq\frac{1}%
{2}W^{\prime\prime}(a)(1+\eta)(s-a)^{2} \label{2d taylor}%
\end{equation}
for all $a\leq s\leq a+\delta_{\eta}$.

\begin{theorem}
\label{theorem 1d properties minimizers}Assume that $W$ satisfies
\eqref{W_Smooth}-\eqref{W' three zeroes}, that $\alpha_{-},\beta_{-}$ satisfy \eqref{alpha and beta minus}, and that $\omega$ satisfies \eqref{etaSmooth} and is strictly increasing with $\omega^{\prime
}(0)>0$. Let $a\leq\alpha_{\varepsilon},\,\beta_{\varepsilon}\leq b$ satisfy
\eqref{alpha epsilon and beta epsilon} and let $v_{\varepsilon}$ be the
minimizer of $G_{\varepsilon}$ obtained in Theorem \ref{theorem 1d EL}. Given
$k\in\mathbb{N}$ with $k\geq\gamma$, there exist $0<\varepsilon_{0}<1$, $C>0$
depending only on $\alpha_{-}$, $\beta_{-}$, $k$, $A_{0}$, $B_{0}$, $T$,
$\omega$, $W$, such that, for all $0<\varepsilon<\varepsilon_{0}$, the
following properties hold:

\begin{enumerate}
\item[(i)] If $T_{\varepsilon}$ is the first time such that $v_{\varepsilon
}=\beta_{\varepsilon}-\varepsilon^{k}$, then
\begin{equation}
T_{\varepsilon}\leq C\varepsilon|\log\varepsilon|. \label{1d T epsilon}%
\end{equation}

\item[(ii)] Let $0<\eta<\frac{1}{4}$, let $\delta_{\eta}$ be as in
\eqref{2d taylor}, and let $S_{\varepsilon,\eta}$ be the first time such that
$v_{\varepsilon}=a+\delta_{\eta}$. Then there exists a constant $C_{\eta}>0$,
depending on $\eta$, $\alpha_{-}$, $\beta_{-}$, $k$, $A_{0}$, $B_{0}$, $T$,
$\omega$, $W$, such that%
\begin{equation}
S_{\varepsilon,\eta}\geq\frac{1}{2^{1/2}(W^{\prime\prime}(a))^{1/2}}\left(
\varepsilon|\log\varepsilon|-\eta\varepsilon|\log\varepsilon|\right)
-C_{\eta}\varepsilon. \label{1d S epsilon}%
\end{equation}

\end{enumerate}
\end{theorem}

\begin{proof}
In this proof, $\varepsilon_{0}$ and the constants $C$, $C_{0}$, and $C_{1}$
depend only on $\alpha_{-}$, $\beta_{-}$, $A_{0}$, $B_{0}$, $T$, $\omega$,
$W$. In what follows, we will take $\varepsilon_{0}$ smaller and $C$, $C_{0}$,
and $C_{1}$ larger, if necessary, preserving the same dependence on the parameters.

By Theorem \ref{theorem 1d limsup alpha=a},%
\begin{equation}
G_{\varepsilon}^{(1)}(v_{\varepsilon})\leq C_{W}\omega(0)+C_{1}\varepsilon
|\log\varepsilon| \label{2d limsup bound}%
\end{equation}
for all $0<\varepsilon<\varepsilon_{0}$.

Let
\[
t_{1}^{\varepsilon}<t_{2}^{\varepsilon}<t_{3}^{\varepsilon}<t_{4}%
^{\varepsilon}%
\]
be the first time such that $v_{\varepsilon}$ equals $\alpha_{\varepsilon
}+\varepsilon^{k}$, $\alpha_{-}$, $\beta_{-}$, and $\beta_{\varepsilon
}-\varepsilon^{k}$, respectively.

\textbf{Step 1: }We claim that there exist $0<\varepsilon_{0}<1$ and $C>0$
such that%
\begin{equation}
t_{2}^{\varepsilon}-t_{1}^{\varepsilon}\leq C\varepsilon|\log\varepsilon|
\label{2d t2-t1}%
\end{equation}
for all $0<\varepsilon<\varepsilon_{0}$. To see this, observe that since
$v_{\varepsilon}(0)=\alpha_{\varepsilon}<\alpha_{\varepsilon}+\varepsilon^{k}%
$, we have that $v_{\varepsilon}^{\prime}(t_{1}^{\varepsilon})\geq0$. Using
(\ref{W' three zeroes}) and (\ref{1d euler lagrange}),%
\[
2\varepsilon^{2}(v_{\varepsilon}^{\prime}(t)\omega(t))^{\prime}=W^{\prime
}(v_{\varepsilon}(t))\omega(t)>0
\]
for all $a<v_{\varepsilon}(t)<c$. In particular, since $\alpha_{-}<c$, we have
that $v_{\varepsilon}^{\prime}(t)>0$ for all $t_{1}^{\varepsilon}<t\leq
t_{2}^{\varepsilon}$. It follows that $[t_{1}^{\varepsilon},t_{2}%
^{\varepsilon}]$ is a maximal interval of the set $A_{\varepsilon}$ defined in
(\ref{1d set A epsilon}), and so by Theorem \ref{theorem barrier 1d}, the
claim (\ref{2d t2-t1}) follows.

\textbf{Step 2: }We claim that there exist $0<\varepsilon_{0}<1$ and $C>0$
such that%
\begin{equation}
t_{3}^{\varepsilon}-t_{2}^{\varepsilon}\leq C\varepsilon\label{2d t3-t2}%
\end{equation}
for all $0<\varepsilon<\varepsilon_{0}$. Indeed, since $v_{\varepsilon
}^{\prime}(t_{2}^{\varepsilon})>0$, by (\ref{1d v' near a}) and
(\ref{1d v' near b}), we have that $v_{\varepsilon}^{\prime}(t)>0$ for all
$t\geq t_{2}^{\varepsilon}$ such that $v_{\varepsilon}(t)\leq b-\tau
_{0}\varepsilon^{1/2}$. It follows that $\alpha_{-}\leq v_{\varepsilon}%
(t)\leq\beta_{-}$ for all $t\in\lbrack t_{2}^{\varepsilon},t_{3}^{3}]$. Since
$\omega$ is increasing, by (\ref{2d limsup bound}),%
\[
C\geq G_{\varepsilon}(v_{\varepsilon})\geq\frac{\omega(0)}{\varepsilon}%
\int_{t_{2}^{\varepsilon}}^{t_{3}^{\varepsilon}}W(v_{\varepsilon})\,dt\geq
\min_{\lbrack\alpha_{-},\beta_{-}]}W\frac{\omega(0)}{\varepsilon}%
(t_{3}^{\varepsilon}-t_{2}^{\varepsilon}),
\]
which proves (\ref{2d t3-t2}).

\textbf{Step 3: }We claim that there exist $0<\varepsilon_{0}<1$ and $C>0$
such that%
\begin{equation}
t_{4}^{\varepsilon}-t_{3}^{\varepsilon}\leq C\varepsilon|\log\varepsilon|
\label{2d t4-t3}%
\end{equation}
for all $0<\varepsilon<\varepsilon_{0}$. Since $v_{\varepsilon}^{\prime}(t)>0$
for all $t\geq t_{3}^{\varepsilon}$ such that $v_{\varepsilon}(t)\leq
b-\tau_{0}\varepsilon^{1/2}$, there are two possible scenarios. Either
$v_{\varepsilon}(t)\geq\beta_{-}$ for all $t\in\lbrack t_{3}^{\varepsilon
},t_{4}^{\varepsilon}]$, in which case (\ref{2d t4-t3}) follows from Theorem
\ref{theorem barrier 1d}, or there exists a last time $t_{3}^{\varepsilon
}<t_{\varepsilon}<t_{4}^{\varepsilon}$ such that $v_{\varepsilon}=\beta_{-}%
$and $\beta_{-}\leq v_{\varepsilon}(t)\leq t_{4}^{\varepsilon}$. We claim that
this latter case cannot happen.

Since $v_{\varepsilon}^{\prime}(t)>0$ for all $t\geq t_{3}^{\varepsilon}$ such
that $v_{\varepsilon}(t)\leq b-\tau_{0}\varepsilon^{1/2}$, there exists
$\tau_{\varepsilon}\in(t_{3}^{\varepsilon},t_{\varepsilon})$ such that
$v_{\varepsilon}(\tau_{\varepsilon})=b-\tau_{0}\varepsilon^{1/2}$. It follows
that $v_{\varepsilon}([t_{3}^{\varepsilon},\tau_{\varepsilon}])=[\beta
_{-},b-\tau_{0}\varepsilon^{1/2}]$, while $v_{\varepsilon}([t_{\varepsilon
},t_{4}^{\varepsilon}])=[\beta_{-},b-\tau_{0}\varepsilon^{1/2}]$. Then by
(\ref{2d limsup bound}), and (\ref{eta close to eta0}), and the fact that
$\omega$ is increasing%
\begin{align*}
C_{W}\omega(0)+C_{1}\varepsilon|\log\varepsilon|  &  \geq G_{\varepsilon
}^{(1)}(v_{\varepsilon})\geq G_{\varepsilon}^{(1)}(v;[0,\tau_{\varepsilon
}]\cup\lbrack t_{\varepsilon},t_{4}^{\varepsilon}])\\
&  \geq\omega(0)\int_{[0,\tau_{\varepsilon}]\cup\lbrack t_{\varepsilon}%
,t_{4}^{\varepsilon}]}2W^{1/2}(v_{\varepsilon})|v_{\varepsilon}^{\prime
}|\,dt\\
&  =\left(  \operatorname*{d}\nolimits_{W}(\alpha_{\varepsilon},b-\tau
_{0}\varepsilon^{1/2})+\operatorname*{d}\nolimits_{W}(\beta_{-},\beta
_{\varepsilon}-\varepsilon^{k})\right)  \omega(0).
\end{align*}
Using the fact that $\operatorname*{d}\nolimits_{W}(\cdot,r)$ and
$\operatorname*{d}\nolimits_{W}(s,\cdot)$ are Lipschitz continuous and
(\ref{alpha epsilon and beta epsilon}), it follows that%
\[
C_{W}\omega(0)+C_{1}\varepsilon|\log\varepsilon|\geq(C_{W}+\operatorname*{d}%
\nolimits_{W}(\beta_{-},b)-L(A_{0}\varepsilon^{\gamma}+2\tau_{0}%
\varepsilon^{1/2}))\omega(0),
\]
or, equivalently,
\[
C(\varepsilon|\log\varepsilon|+\varepsilon^{\gamma}+\varepsilon^{1/2}%
)\geq\operatorname*{d}\nolimits_{W}(\beta_{-},b)\omega(0),
\]
which is a contradiction provided we take $0<\varepsilon<\varepsilon_{0}$ with
$\varepsilon_{0}$ sufficiently small.

\textbf{Step 4: }We claim that there exist $0<\varepsilon_{0}<1$ and $C_{0}>0$
such that%
\begin{equation}
t_{1}^{\varepsilon}\leq C_{0}\varepsilon|\log\varepsilon|. \label{2d t1}%
\end{equation}
Fix $C_{0}>0$ such that
\begin{equation}
\frac{1}{2}\omega^{\prime}(0)C_{0}C_{W}>2C_{1}, \label{2d C0}%
\end{equation}
where $C_{1}$ is the constant in (\ref{2d limsup bound}) and let
$0<\varepsilon<\varepsilon_{0}$, where $\varepsilon_{0}$ was .... Assume by
contradiction that
\[
t_{1}^{\varepsilon}>C_{0}\varepsilon|\log\varepsilon|=:t_{0}^{\varepsilon}.
\]
Since $\omega$ is increasing, we have
\begin{align*}
G_{\varepsilon}(v_{\varepsilon})  &  \geq\int_{t_{1}^{\varepsilon}}^{T}\left(
\frac{1}{\varepsilon}W(v_{\varepsilon})+\varepsilon(v_{\varepsilon}^{\prime
})^{2}\right)  \omega\,dt\geq\omega(t_{0}^{\varepsilon})\int_{t_{1}%
^{\varepsilon}}^{T}\left(  \frac{1}{\varepsilon}W(v_{\varepsilon}%
)+\varepsilon(v_{\varepsilon}^{\prime})^{2}\right)  \,dt\\
&  \geq\omega(t_{0}^{\varepsilon})\int_{\alpha_{\varepsilon}+\varepsilon^{k}%
}^{\beta_{\varepsilon}}2W^{1/2}(s)\,ds\geq\omega(t_{0}^{\varepsilon}%
)(C_{W}-C\varepsilon^{2\gamma}).
\end{align*}
By Taylor's formula, for $0<t<t_{0}$ for some $t_{0}$ small.
\[
\omega(t_{0}^{\varepsilon})=\omega(0)+\omega^{\prime}(0)t_{0}^{\varepsilon
}+o(t_{0}^{\varepsilon})\geq\omega(0)+\frac{1}{2}\omega^{\prime}%
(0)t_{0}^{\varepsilon}%
\]
for all $0<\varepsilon<\varepsilon_{0}$ provided $\varepsilon_{0}$ is taken
even smaller (depending on $C_{0}$). Then by (\ref{2d limsup bound}),
\[
C_{W}\omega(0)+C_{1}\varepsilon|\log\varepsilon|\geq G_{\varepsilon}%
^{(1)}(v_{\varepsilon})\geq\left(  \omega(0)+\frac{1}{2}\omega^{\prime
}(0)C_{0}\varepsilon|\log\varepsilon|\right)  (C_{W}-C\varepsilon^{2\gamma}),
\]
and so
\[
\frac{1}{2}\omega^{\prime}(0)C_{0}C_{W}\varepsilon|\log\varepsilon|\leq
C_{1}\varepsilon|\log\varepsilon|+\left(  \omega(0)+\frac{1}{2}\omega^{\prime
}(0)C_{0}\varepsilon|\log\varepsilon|\right)  C\varepsilon^{2\gamma}\leq
2C_{1}\varepsilon|\log\varepsilon|
\]
provided $\varepsilon_{0}$ is taken even smaller (depending on $C_{0}$). This
contradicts (\ref{2d C0}).

Combining Steps 1-4 proves (\ref{1d T epsilon}).

\textbf{Step 5: }In this step, we prove item (ii). Rewrite
(\ref{1d euler lagrange}) as
\begin{equation}
2\varepsilon^{2}v_{\varepsilon}^{\prime\prime}(t)-W^{\prime}(v_{\varepsilon
}(t))+2\varepsilon^{2}\frac{\omega^{\prime}(t)}{\omega(t)}v_{\varepsilon
}^{\prime}(t)=0. \label{1d euler lagrange expanded}%
\end{equation}
Multiply (\ref{1d euler lagrange expanded}) by $\frac{1}{\varepsilon
}v_{\varepsilon}^{\prime}(t)$ to get%
\begin{equation}
\varepsilon((v_{\varepsilon}^{\prime}(t))^{2})^{\prime}-\frac{1}{\varepsilon
}(W(v_{\varepsilon}(t)))^{\prime}+2\varepsilon\frac{\omega^{\prime}(t)}%
{\omega(t)}(v_{\varepsilon}^{\prime}(t))^{2}=0. \label{1d 71}%
\end{equation}
Integrating between $0$ and $t$ and we have%
\begin{equation}
\varepsilon(v_{\varepsilon}^{\prime}(t))^{2}-\frac{1}{\varepsilon
}W(v_{\varepsilon}(t))+2\varepsilon\int_{0}^{t}\frac{\omega^{\prime}}{\omega
}(v_{\varepsilon}^{\prime})^{2}dt=\varepsilon(v_{\varepsilon}^{\prime}%
(0))^{2}-\frac{1}{\varepsilon}W(\alpha_{\varepsilon}). \label{2d 71}%
\end{equation}
By \eqref{ineq:initial_value_veps}  and the fact that $\omega^{\prime}\geq0$, we have%
\[
\varepsilon(v_{\varepsilon}^{\prime}(t))^{2}-\frac{1}{\varepsilon
}W(v_{\varepsilon}(t))\leq C.
\]
Hence,%
\begin{equation}
\varepsilon^{2}(v_{\varepsilon}^{\prime}(t))^{2}\leq W(v_{\varepsilon
}(t))+C\varepsilon\label{2d 71a}%
\end{equation}
for all $t\in I$. Let $\delta_{\eta}$ be as in (\ref{2d taylor}) and let
$S_{\varepsilon,\eta}$ be the first time such that $v_{\varepsilon}%
=a+\delta_{\eta}$. Then, by (\ref{2d taylor}),%
\begin{align*}
\varepsilon^{2}(v_{\varepsilon}^{\prime}(t))^{2}  &  \leq\frac{1}{2}%
W^{\prime\prime}(a)(1+\eta)(v_{\varepsilon}(t)-a)^{2}+C\varepsilon\\
&  \leq\frac{1}{2}W^{\prime\prime}(a)(1+2\eta)(v_{\varepsilon}(t)-a)^{2}%
\end{align*}
provided $a+c_{\eta}\varepsilon^{1/2}\leq v_{\varepsilon}(t)\leq
a+\delta_{\eta}$ and $t\leq$\thinspace$S_{\varepsilon,\eta}$, where%
\[
c_{\eta}:=\left(  \frac{2C}{W^{\prime\prime}(a)\eta}\right)^{1/2}.
\]
In turn,
\[
\frac{\varepsilon v_{\varepsilon}^{\prime}(t)}{v_{\varepsilon}(t)-a}%
\leq\left(  \frac{1}{2}W^{\prime\prime}(a)(1+2\eta)\right)  ^{1/2}:=L_{\eta}.
\]
Let $R_{\varepsilon,\eta}$ be the first time such that $v_{\varepsilon
}=a+c_{\eta}\varepsilon^{1/2}$. Note that by (\ref{1d v' near a}),
$v_{\varepsilon}^{\prime}(t)\neq0$ whenever $a+\tau_{0}\varepsilon^{1/2}\leq
v_{\varepsilon}(t)\leq\beta_{-}$. By taking $\eta$ smaller, if necessary, we
can assume that $c_{\eta}\geq\tau_{0}$. Hence, $v_{\varepsilon}(t)\in\lbrack
a+c_{\eta}\varepsilon^{1/2},a+\delta_{\eta}]$ for all $t\in\lbrack
R_{\varepsilon,\eta},S_{\varepsilon,\eta}]$. Integrating in $[R_{\varepsilon
,\eta},S_{\varepsilon,\eta}]$ and using the change of variables $\rho
=v_{\varepsilon}(t)$ gives%
\[
\varepsilon\log\delta_{\eta}-\varepsilon\log(c_{\eta}\varepsilon^{1/2}%
)=\int_{R_{\varepsilon,\eta}}^{S_{\varepsilon,\eta}}\frac{\varepsilon
v_{\varepsilon}^{\prime}(t)}{v_{\varepsilon}(t)-a}\,dt\leq L_{\eta
}(S_{\varepsilon,\eta}-R_{\varepsilon,\eta}).
\] Therefore,
\[
\frac{1}{2}\varepsilon|\log\varepsilon|+\varepsilon\log\delta_{\eta
}-\varepsilon\log c_{\eta}\leq L_{\eta}(S_{\varepsilon,\eta}-R_{\varepsilon
,\eta}).
\]
\hfill
\end{proof}

\begin{corollary}
\label{corollary bound from below v prime}Assume that $W$ satisfies
\eqref{W_Smooth}-\eqref{W' three zeroes}, that $\alpha_{-},\beta_{-}$ satisfy \eqref{alpha and beta minus}, and that $\omega$ satisfies \eqref{etaSmooth} and is strictly increasing with $\omega^{\prime
}(0)>0$. Let $a\leq\alpha_{\varepsilon},\,\beta_{\varepsilon}\leq b$ satisfy
\eqref{alpha epsilon and beta epsilon} and let $v_{\varepsilon}$ be the
minimizer of $G_{\varepsilon}$ obtained in Theorem \ref{theorem 1d EL}. There
exist $0<\varepsilon_{0}<1$, $C>0$ depending only on $\alpha_{-}$, $\beta_{-}%
$, $A_{0}$, $B_{0}$, $T$, $\omega$, $W$, such that%
\[
|v_{\varepsilon}^{\prime}(t)|\geq\frac{C}{\varepsilon^{1/2}}%
\]
for all $0<\varepsilon<\varepsilon_{0}$ and for all $t$ such that
$\alpha_{\varepsilon}\leq v_{\varepsilon}(t)\leq a+\tau_{0}\varepsilon^{1/2}$,
where $\tau_{0}$ is the constant given in Theorem \ref{theorem monotonicity}.
\end{corollary}

\begin{proof}
In this proof, $\varepsilon_{0}$ and $C$ depend only on $\alpha_{-}$,
$\beta_{-}$, $A_{0}$, $B_{0}$, $T$, $\omega$, $W$. Since $\omega$ is
increasing%
\[
\int_{0}^{T_{\varepsilon}}\left(  \frac{1}{\varepsilon}W(v_{\varepsilon
})+\varepsilon(v_{\varepsilon}^{\prime})^{2}\right)  \omega\,dt\geq
2\omega(0)\int_{0}^{T_{\varepsilon}}W^{1/2}(v_{\varepsilon})v_{\varepsilon
}^{\prime}\,dt=2\omega(0)\int_{\alpha_{\varepsilon}}^{\beta_{\varepsilon
}-\varepsilon^{k}}W^{1/2}(\rho)\,d\rho
\]
and so%
\begin{align*}
\int_{0}^{T_{\varepsilon}}\left(  \frac{1}{\varepsilon}W(v_{\varepsilon
})+\varepsilon(v_{\varepsilon}^{\prime})^{2}\right)  \omega\,dt-C_{W}%
\omega(0)  &  \geq-2\omega(0)\int_{\beta_{\varepsilon}-\varepsilon^{k}}%
^{\beta_{k}}W^{1/2}(\rho)\,d\rho-2\omega(0)\int_{a}^{\alpha_{\varepsilon}%
}W^{1/2}(\rho)\,d\rho\\
&  \geq-C\varepsilon^{2\gamma}.
\end{align*}
Hence, also by (\ref{2d limsup bound}),%
\begin{align*}
\int_{T_{\varepsilon}}^{T}\left(  \frac{1}{\varepsilon}W(v_{\varepsilon
})+\varepsilon(v_{\varepsilon}^{\prime})^{2}\right)  \omega\,dt  &  =\int%
_{0}^{T}\left(  \frac{1}{\varepsilon}W(v_{\varepsilon})+\varepsilon
(v_{\varepsilon}^{\prime})^{2}\right)  \omega\,dt-\int_{0}^{\,T_{\varepsilon}%
}\left(  \frac{1}{\varepsilon}W(v_{\varepsilon})+\varepsilon(v_{\varepsilon
}^{\prime})^{2}\right)  \omega\,dt\\
&  \leq C_{W}\omega(0)+C_{1}\varepsilon|\log\varepsilon|-C_{W}\omega
(0)+C\varepsilon^{2\gamma}\leq C\varepsilon|\log\varepsilon|
\end{align*}
since $\gamma>1$. Therefore,
\[
\int_{T_{\varepsilon}}^{T}\left(  \frac{1}{\varepsilon}W(v_{\varepsilon
})+\varepsilon(v_{\varepsilon}^{\prime})^{2}\right)  \omega\,dt\leq
C\varepsilon|\log\varepsilon|.
\]
Since
\[
\frac{2}{T}\int_{T/2}^{T}\left(  \frac{1}{\varepsilon}W(v_{\varepsilon
})+\varepsilon(v_{\varepsilon}^{\prime})^{2}\right)  dt\leq\frac{2}%
{T\min\omega}\int_{\,T_{\varepsilon}}^{T}\left(  \frac{1}{\varepsilon
}W(v_{\varepsilon})+\varepsilon(v_{\varepsilon}^{\prime})^{2}\right)
\omega\,dt\leq C\varepsilon|\log\varepsilon|,
\]
by the mean value theorem, there exists $t_{\varepsilon}$ such that
\[
\frac{1}{\varepsilon}W(v_{\varepsilon}(t_{\varepsilon}))+\varepsilon
(v_{\varepsilon}^{\prime}(t_{\varepsilon}))^{2}=\frac{2}{T}\int_{T/2}%
^{T}\left(  \frac{1}{\varepsilon}W(v_{\varepsilon})+\varepsilon(v_{\varepsilon
}^{\prime})^{2}\right)  dt\leq C\varepsilon|\log\varepsilon|.
\]
Integrating (\ref{1d 71}) from $t$ to $t_{\varepsilon}$ we get%
\begin{align}
\varepsilon(v_{\varepsilon}^{\prime}(t))^{2}-\frac{1}{\varepsilon
}W(v_{\varepsilon}(t_{\varepsilon}))  &  =\varepsilon(v_{\varepsilon}^{\prime
}(t_{\varepsilon}))^{2}-\frac{1}{\varepsilon}W(v_{\varepsilon}(t_{\varepsilon
}))+\int_{t}^{t_{\varepsilon}}2\varepsilon(v_{\varepsilon}^{\prime})^{2}%
\frac{\omega^{\prime}}{\omega}\,dr\nonumber\\
&  \geq-C\varepsilon|\log\varepsilon|+\int_{t}^{t_{\varepsilon}}%
2\varepsilon(v_{\varepsilon}^{\prime})^{2}\frac{\omega^{\prime}}{\omega}\,dr.
\label{2d 42}%
\end{align}
Since $\omega^{\prime}(0)>0$ and $\omega^{\prime}$ is continuous, there exists
$\tau_{0}>0$ such that $\omega^{\prime}(t)\geq\frac{1}{2}\omega^{\prime}(0)$
for all $0<t\leq\tau_{0}$. By taking $\varepsilon_{0}$ even smaller, we can
assume that $T_{\varepsilon}<\tau_{0}$ for all $0<\varepsilon<\varepsilon_{0}%
$. It follows that%
\begin{equation}
\int_{t}^{t_{\varepsilon}}2\varepsilon(v_{\varepsilon}^{\prime})^{2}%
\frac{\omega^{\prime}}{\omega}\,dr\geq\varepsilon\frac{\omega^{\prime}%
(0)}{\min\omega}\int_{t}^{T_{\varepsilon}}2\varepsilon(v_{\varepsilon}%
^{\prime})^{2}\,dt. \label{2d 43}%
\end{equation}
Let $P_{\varepsilon}$ be the first time such that $v_{\varepsilon}=a+\tau
_{0}\varepsilon^{1/2}$, where $\tau_{0}$ is the constant in Theorem
\ref{theorem monotonicity}. Then by Theorem \ref{theorem monotonicity}, and
the properties of $W$,
\[
\varepsilon(v_{\varepsilon}^{\prime})^{2}\geq\frac{1}{2}\sigma^{2}%
(v_{\varepsilon}-a)^{2}\geq CW(v_{\varepsilon})
\]
for all $t$ such that $a+\tau_{0}\varepsilon^{1/2}\leq v_{\varepsilon}\leq c$.
Let $J\subseteq\lbrack P_{\varepsilon},T_{\varepsilon}]$ be a maximal interval
such that $a+\tau_{0}\varepsilon^{1/2}\leq v_{\varepsilon}\leq c$. It follows
that%
\begin{align*}
\int_{t}^{T_{\varepsilon}}2\varepsilon(v_{\varepsilon}^{\prime})^{2}%
\omega\,dr  &  \geq\int_{J}[\varepsilon(v_{\varepsilon}^{\prime}%
)^{2}+CW(v_{\varepsilon})]\,dr\geq C\int_{J}2W^{1/2}(v_{\varepsilon
})v_{\varepsilon}^{\prime}\,dr\\
&  =C\int_{a+\tau_{0}\varepsilon^{1/2}}^{c}2W^{1/2}(s)\,ds\geq C\int%
_{\frac{a+c}{2}}^{c}2W^{1/2}(s)\,ds=:C_{1}%
\end{align*}
for all $0\leq t\leq P_{\varepsilon}$. In turn, from (\ref{2d 42}) and
(\ref{2d 43}),%
\[
\varepsilon(v_{\varepsilon}^{\prime}(t))^{2}-\frac{1}{\varepsilon
}W(v_{\varepsilon}(t))\geq-C\varepsilon|\log\varepsilon|+\frac{\omega^{\prime
}(0)}{\min\omega}C_{1}.
\]
Hence,
\[
\varepsilon(v_{\varepsilon}^{\prime}(t))^{2}\geq C_{2}>0
\]
for all $0\leq t\leq P_{\varepsilon}$.\hfill
\end{proof}

\begin{remark}
\label{remark increasing}Note that this corollary, together with Theorem
\ref{theorem monotonicity}, implies that $v_{\varepsilon}^{\prime}$ does not
vanish as long as $v_{\varepsilon}(t)\leq b-\tau_{0}\varepsilon^{1/2}$. Hence
$v_{\varepsilon}^{\prime}(t)>0$ as long as $v_{\varepsilon}(t)\leq b-\tau
_{0}\varepsilon^{1/2}$.
\end{remark}

\subsection{Second-Order $\Gamma$-liminf} 

In this subsection, we present the $\Gamma$-liminf  counterparts of Theorems
\ref{theorem 1d limsup a<alpha} and \ref{theorem 1d limsup alpha=a}. We recall
that when $a<\alpha$, $G_{\varepsilon}^{(2)}$ is defined as in
(\ref{G 2 epsilon a<alpha}). For the proof of the following theorem, we refer
to \cite[Theorem 3.15]{fonseca-kreutz-leoni2025I}.

\begin{theorem}
[Second-Order Liminf, $a<\alpha$]\label{theorem liminf 1d a<alpha}Assume that
$W$ satisfies \eqref{W_Smooth}-\eqref{W' three zeroes}, that that
$\alpha_{-}$ satisfies \eqref{alpha and beta minus}, and that
$\omega$ satisfies \eqref{etaSmooth}, \eqref{eta close to eta0},
and \eqref{eta 0 alpha-}. Let $\alpha_{-}\leq\alpha_{\varepsilon}%
,\,\beta_{\varepsilon}\leq b$ satisfy \eqref{alpha epsilon and beta epsilon}
and let $v_{\varepsilon}$ be the minimizer of $G_{\varepsilon}$ obtained in
Theorem \ref{theorem 1d EL}. Then there exist $0<\varepsilon_{0}<1$, $C>0$,
and $l_{0}>1$, depending only on $\alpha_{-}$, $A_{0}$, $B_{0}$, $T$, $\omega
$, and $W$, such that%
\[
G_{\varepsilon}^{(2)}(v_{\varepsilon})\geq2\omega^{\prime}(0)\int_{0}%
^{l}W^{1/2}(w_{\varepsilon})w_{\varepsilon}^{\prime}s\,ds-Ce^{-l\mu}\left(
l\mu+1\right)  -C\varepsilon^{1/2}l-C\varepsilon^{\gamma_{1}}|\log
\varepsilon|^{2+\gamma_{0}}%
\]
for all $0<\varepsilon<\varepsilon_{0}$ and $l>l_{0}$, where $G_{\varepsilon
}^{(2)}$ is defined in \eqref{G 2 epsilon a<alpha}, $w_{\varepsilon
}(s):=v_{\varepsilon}(\varepsilon s)$ for $s\in\lbrack0,T\varepsilon^{-1}]$
satisfies%
\[
\lim_{\varepsilon\rightarrow0^{+}}\int_{0}^{l}W^{1/2}(w_{\varepsilon
})w_{\varepsilon}^{\prime}s\,ds=\int_{0}^{l}W^{1/2}(z_{\alpha})z_{\alpha
}^{\prime}s\,ds
\]
for every $l>0$, and where $z_{\alpha}$ solves the Cauchy problem
\eqref{cauchy problem z alpha}. In particular,%
\[
\liminf_{\varepsilon\rightarrow0^{+}}G_{\varepsilon}^{(2)}(v_{\varepsilon
})\geq2\omega^{\prime}(0)\int_{0}^{\infty}W^{1/2}(z_{\alpha})z_{\alpha
}^{\prime}s\,ds.
\]

\end{theorem}

When $\alpha=a$, $G_{\varepsilon}^{(2)}$ is defined as in
(\ref{G 2 epsilon a=alpha}).

\begin{theorem}
[Second-Order Liminf, $a=\alpha$]\label{theorem liminf 1d alpha=a}Assume that
$W$ satisfies \eqref{W_Smooth}-\eqref{W' three zeroes}, that
$\alpha_{-},\beta_{-}$ satisfy \eqref{alpha and beta minus}, and
that $\omega$ satisfies \eqref{etaSmooth} and is strictly
increasing with $\omega^{\prime}(0)>0$. Let $a\leq\alpha_{\varepsilon}%
,\,\beta_{\varepsilon}\leq b$ satisfy \eqref{alpha epsilon and beta epsilon}
and let $v_{\varepsilon}$ be the minimizer of $G_{\varepsilon}$ obtained in
Theorem \ref{theorem 1d EL}. Then for every $0<\eta<\frac{1}{4}$ there exist a
constant $C_{\eta}>0$, depending on $\eta$, $\alpha_{-}$, $\beta_{-}$, $A_{0}%
$, $B_{0}$, $T$, $\omega$, $W$, such that
\[
G_{\varepsilon}^{(2)}(v_{\varepsilon})\geq\frac{C_{W}\omega^{\prime}%
(0)}{2^{1/2}(W^{\prime\prime}(a))^{1/2}}\left(  1-\eta\right)  -\frac{C_{\eta
}}{|\log\varepsilon|}%
\]
for all $0<\varepsilon<\varepsilon_{\eta}$, where $\varepsilon_{\eta}>0$
depends on $\eta$, $\alpha_{-}$, $\beta_{-}$, $A_{0}$, $B_{0}$, $T$, $\omega$,
$W$, and where $G_{\varepsilon}^{(2)}$ is defined in \eqref{G 2 epsilon a=alpha}.
\end{theorem}

\begin{proof}
In this proof, the constants $\varepsilon_{0}$, $C$, and $C_{1}$ depend only
on $\alpha_{-}$, $\beta_{-}$, $A_{0}$, $B_{0}$, $T$, $\omega$, $W$, while
$\varepsilon_{\eta}$ and $C_{\eta}$ depend on all these parameters but also on
$\eta$. Since $\omega\in C^{1,d}(I)$, by Taylor's formula, for $t\in
\lbrack0,T]$,%
\[
\omega(t)=\omega(0)+\omega^{\prime}(0)t+R_{1}(t),
\]
where%
\begin{equation}
|R_{1}(t)|=|\omega^{\prime}(\theta t)-\omega^{\prime}(0)|t\leq|\omega^{\prime
}|_{C^{0,d}}t^{1+d}. \label{2d R1}%
\end{equation}
Let $T_{\varepsilon}$ be the first time such that $v_{\varepsilon}%
=\beta_{\varepsilon}-\varepsilon^{k}$, and write%
\begin{align}
G_{\varepsilon}^{(2)}(v_{\varepsilon})  &  =\left[  \int_{0}^{T_{\varepsilon}%
}\left(  \frac{1}{\varepsilon}W(v_{\varepsilon})+\varepsilon(v_{\varepsilon
}^{\prime})^{2}\right)  \,dt-C_{W}\right]  \frac{\omega(0)}{\varepsilon
|\log\varepsilon|}\nonumber\\
&  \quad+\int_{0}^{T_{\varepsilon}}\left(  \frac{1}{\varepsilon}%
W(v_{\varepsilon})+\varepsilon(v_{\varepsilon}^{\prime})^{2}\right)
t\,dt\frac{\omega^{\prime}(0)}{\varepsilon|\log\varepsilon|}%
\label{2d liminf energy}\\
&  \quad+\int_{0}^{T_{\varepsilon}}\left(  \frac{1}{\varepsilon}%
W(v_{\varepsilon})+\varepsilon(v_{\varepsilon}^{\prime})^{2}\right)
R_{1}\,dt\frac{1}{\varepsilon|\log\varepsilon|}\nonumber\\
&  \quad+\int_{T_{\varepsilon}}^{T}\left(  \frac{1}{\varepsilon}%
W(v_{\varepsilon})+\varepsilon(v_{\varepsilon}^{\prime})^{2}\right)
\omega\,dt\frac{1}{\varepsilon|\log\varepsilon|}=:\mathcal{A}+\mathcal{B}%
+\mathcal{C}+\mathcal{D}.\nonumber
\end{align}
\textbf{Step 1: }We estimate $\mathcal{A}$. By the change of variables $\rho
=$\thinspace$v_{\varepsilon}(t)$, (\ref{W near a}), (\ref{W near b}), and
(\ref{alpha epsilon and beta epsilon}), we have
\begin{align}
\int_{0}^{T_{\varepsilon}}\left(  \frac{1}{\varepsilon}W(v_{\varepsilon
})+\varepsilon(v_{\varepsilon}^{\prime})^{2}\right)  \,dt  &  \geq2\int%
_{0}^{T_{\varepsilon}}2W^{1/2}(v_{\varepsilon})v_{\varepsilon}^{\prime
}\,dt=2\int_{\alpha_{\varepsilon}}^{\beta_{\varepsilon}-\varepsilon^{k}%
}W^{1/2}(\rho)\,d\rho\nonumber\\
&  =C_{W}-2\int_{a}^{\alpha_{\varepsilon}}W^{1/2}(\rho)\,d\rho-2\int%
_{\beta_{\varepsilon}-\varepsilon^{k}}^{b}W^{1/2}(\rho)\,d\rho\label{2d 41}\\
&  =C_{W}-C\varepsilon^{2\gamma}\nonumber
\end{align}
for all $0<\varepsilon<\varepsilon_{0}$. Hence,
\[
\mathcal{A}\geq-C\frac{\varepsilon^{2\gamma-1}}{|\log\varepsilon|}%
\]
for all $0<\varepsilon<\varepsilon_{0}$.

\textbf{Step 2: }We estimate $\mathcal{B}$ in (\ref{2d liminf energy}). Let
$0<\eta<\frac{1}{4}$, let $\delta_{\eta}$ be as in \eqref{2d taylor}, and let
$S_{\varepsilon,\eta}$ be the first time such that $v_{\varepsilon}%
=a+\delta_{\eta}$. By the change of variables $t=r+S_{\varepsilon,\eta}$,%
\begin{align*}
\mathcal{B}  & =\int_{-S_{\varepsilon,\eta}}^{T_{\varepsilon}-S_{\varepsilon
,\eta}}\left(  \frac{1}{\varepsilon}W(\bar{v}_{\varepsilon})+\varepsilon
(\bar{v}_{\varepsilon}^{\prime})^{2}\right)  \,dr\frac{\omega^{\prime
}(0)S_{\varepsilon,\eta}}{\varepsilon|\log\varepsilon|}\\&\quad+\int_{-S_{\varepsilon
,\eta}}^{T_{\varepsilon}-S_{\varepsilon,\eta}}\left(  \frac{1}{\varepsilon
}W(\bar{v}_{\varepsilon})+\varepsilon(\bar{v}_{\varepsilon}^{\prime}%
)^{2}\right)  r\,dr\frac{\omega^{\prime}(0)}{\varepsilon|\log\varepsilon|} \\
&  =:\mathcal{B}_{1}+\mathcal{B}_{2},
\end{align*}
where $\bar{v}_{\varepsilon}(r):=v_{\varepsilon}(r+S_{\varepsilon,\eta})$. By
the change of variables $r:=t-S_{\varepsilon,\eta}$, (\ref{2d 41}), and
(\ref{1d S epsilon}),%
\begin{align*}
\mathcal{B}_{1}  &  \geq(C_{W}-C\varepsilon^{2\gamma})\frac{\omega^{\prime
}(0)S_{\varepsilon,\eta}}{\varepsilon|\log\varepsilon|}\\
&  \geq\frac{C_{W}\omega^{\prime}(0)}{2^{1/2}(W^{\prime\prime}(a))^{1/2}%
}\left(  1-\eta\right)  -C\varepsilon^{2\gamma}-\frac{C_{\eta}}{|\log
\varepsilon|}%
\end{align*}
for all $0<\varepsilon<\varepsilon_{\eta}$.

Define%
\[
p_{\varepsilon}(s):=v_{\varepsilon}(\varepsilon s+S_{\varepsilon,\eta}).
\]
Then
\begin{align*}
\mathcal{B}_{2}  &  =\frac{\omega^{\prime}(0)}{|\log\varepsilon|}%
\int_{-\,S_{\varepsilon,\eta}\varepsilon^{-1}}^{(T_{\varepsilon}%
-S_{\varepsilon,\eta})\varepsilon^{-1}}\left(  W(p_{\varepsilon}%
(s))+(p_{\varepsilon}^{\prime}(s))^{2}\right)  s\,ds\\
&  \geq-\frac{\omega^{\prime}(0)}{|\log\varepsilon|}\int_{-\,S_{\varepsilon
,\eta}\varepsilon^{-1}}^{0}\left(  W(p_{\varepsilon}(s))+(p_{\varepsilon
}^{\prime}(s))^{2}\right)  |s|\,ds
\end{align*}
By (\ref{2d 71a}), we have that%
\[
(p_{\varepsilon}^{\prime}(s))^{2}\leq W(p_{\varepsilon}(s))+C\varepsilon.
\]
Hence, by (\ref{1d T epsilon}),%
\begin{align}
\mathcal{B}_{2}  &  \geq-\frac{\omega^{\prime}(0)}{|\log\varepsilon|}%
\int_{-\,S_{\varepsilon,\eta}\varepsilon^{-1}}^{0}2\left(  W(p_{\varepsilon
}(s))+C\varepsilon\right)  |s|\,ds\label{2d 73}\\
&  \geq-\frac{\omega^{\prime}(0)}{|\log\varepsilon|}\int_{-\,S_{\varepsilon
,\eta}\varepsilon^{-1}}^{0}2W(p_{\varepsilon}(s))|s|\,ds-C\varepsilon
|\log\varepsilon|\nonumber
\end{align}
for all $0<\varepsilon<\varepsilon_{\eta}$. By Corollary
\ref{corollary bound from below v prime} and Remark \ref{remark increasing},
\[
v_{\varepsilon}^{\prime}(t)\geq\frac{C}{\varepsilon^{1/2}}%
\]
for all $t$ such that $v_{\varepsilon}(t)\leq a+\tau_{0}\varepsilon^{1/2}$. Therefore,
\[
\varepsilon^{2}(v_{\varepsilon}^{\prime}(t))^{2}\geq C\varepsilon\geq\frac
{C}{\tau_{0}^{2}}(v_{\varepsilon}(t)-a)^{2}.
\]
Together with Theorem \ref{theorem monotonicity}, this implies that%
\[
\varepsilon v_{\varepsilon}^{\prime}(t)\geq\sigma_{0}(v_{\varepsilon}(t)-a)
\]
for all $t\geq0$ such that $v_{\varepsilon}(t)\leq c$, where $\sigma_{0}>0$.
In turn,
\[
p_{\varepsilon}^{\prime}(s)\geq\sigma_{0}(p_{\varepsilon}(s)-a)
\]
for all $-\,S_{\varepsilon,\eta}\varepsilon^{-1}\leq s\leq0$. Hence,
\[
(\log(p_{\varepsilon}(s)-a))^{\prime}=\frac{(p_{\varepsilon}(s)-a)^{\prime}%
}{p_{\varepsilon}(s)-a}\geq\sigma_{0}.
\]
Upon integration, we get%
\[
\log\frac{\delta_{\eta}}{p_{\varepsilon}(s)-a}\geq\sigma_{0}(0-s)
\]
and so%
\[
c-a\geq\delta_{\eta}\geq(p_{\varepsilon}(s)-a)e^{-s\sigma_{0}},
\]
which gives%
\[
0\leq p_{\varepsilon}(s)-a\leq(c-a)e^{\sigma_{0}s}.
\]
In turn, again by (\ref{W near a}), for $s\in\lbrack-L_{\varepsilon
}\varepsilon^{-1},0]$,%
\[
W(p_{\varepsilon}(s))\leq C(p_{\varepsilon}(s)-a)^{2}\leq Ce^{2\sigma_{0}s}.
\]
Hence,%
\[
\int_{-\,S_{\varepsilon,\eta}\varepsilon^{-1}}^{0}2W(p_{\varepsilon
}(s))|s|\,ds\leq C\int_{-\,S_{\varepsilon,\eta}\varepsilon^{-1}}^{0}%
e^{2\sigma_{0}s}|s|\,ds\leq C\int_{-\infty}^{0}e^{2\sigma_{0}s}|s|\,ds.
\]
By (\ref{2d 73}),
\[
\mathcal{B}_{2}\geq-\frac{C}{|\log\varepsilon|}-C\varepsilon|\log\varepsilon|
\]
for all $0<\varepsilon<\varepsilon_{\eta}$. .

\textbf{Step 3: }We estimate $\mathcal{C}$ in (\ref{2d liminf energy}). By
Theorem \ref{theorem 1d limsup alpha=a},%
\[
G_{\varepsilon}^{(1)}(v_{\varepsilon})\leq C_{W}\omega(0)+C_{1}\varepsilon
|\log\varepsilon|
\]
for all $0<\varepsilon<\varepsilon_{0}$. We have,%
\begin{align*}
|\mathcal{C}|  &  \leq C\varepsilon^{d}|\log\varepsilon|^{d}\int%
_{0}^{T_{\varepsilon}}\left(  \frac{1}{\varepsilon}W(v_{\varepsilon
})+\varepsilon(v_{\varepsilon}^{\prime})^{2}\right)  \,dt\\
&  \leq C\varepsilon^{d}|\log\varepsilon|^{d}\frac{1}{\min\omega}\int%
_{0}^{T_{\varepsilon}}\left(  \frac{1}{\varepsilon}W(v_{\varepsilon
})+\varepsilon(v_{\varepsilon}^{\prime})^{2}\right)  \omega\,dt\\
&  \leq C\varepsilon^{d}|\log\varepsilon|^{d}\frac{1}{\min\omega}(C_{W}%
\omega(0)+C_{1}\varepsilon|\log\varepsilon|)\leq C\varepsilon^{d}%
|\log\varepsilon|^{d}%
\end{align*}
for all $0<\varepsilon<\varepsilon_{\eta}$, where we used (\ref{1d T epsilon})
and (\ref{2d R1}).

\textbf{Step 4: }To estimate $\mathcal{D}$ in (\ref{2d liminf energy}),
observe that $\mathcal{D}\geq0$.

Combining the estimates in Steps 1-4 and using (\ref{2d liminf energy}) gives%
\[
G_{\varepsilon}^{(2)}(v_{\varepsilon})\geq\frac{C_{W}\omega^{\prime}%
(0)}{2^{1/2}(W^{\prime\prime}(a))^{1/2}}\left(  1-\eta\right)  -\frac{C_{\eta
}}{|\log\varepsilon|}%
\]
for all $0<\varepsilon<\varepsilon_{\eta}$.\hfill
\end{proof}

\section{Properties of Minimizers of $F_{\varepsilon}$}

\label{section minimizers}In this section,  we study qualitative properties of
critical points and minimizers of the functional $F_{\varepsilon}$ given in
(\ref{functional cahn-hilliard}) and subject to the Dirichlet boundary
conditions (\ref{dirichlet boundary conditions}).

\begin{theorem}
\label{theorem curvature}Assume that $\partial\Omega$ is of class $C^{2}$ and
that $g:\mathbb{R}^{N}\rightarrow\mathbb{R}$ is a function of class $C^{1}$
such that%
\[
a\leq g(x)\leq b\quad\text{for all }x\in\partial\Omega,
\]
and that there exists $x_{0}\in\partial\Omega$ such that%
\[
\kappa(x_{0})>0\quad\text{and}\quad g(x_{0})=a
\]
Then the constant function $b$ is not a minimizer of the functional
$\mathcal{F}^{(1)}$ given in \eqref{firstOrderFormalDefinition}.
\end{theorem}

\begin{proof}
Since the boundary of $\Omega$ is of class $C^{2}$,  without loss of generality by a translation and a
rotation we can assume that $x_{0}=0$ and that
there exist $r_{0}>0$ and a function $f:\mathbb{R}^{N-1}\rightarrow\mathbb{R}$
of class $C^{3}$ such that $f(0)=0$, $\nabla^{\prime}f(0)=0$ and%
\begin{equation}
Q(0,r_{0})\cap\Omega=\{x\in Q(0,r_{0}):\,x_{N}>f(x^{\prime})\}, \label{332}%
\end{equation}
where, with a slight abuse of notation, we are writing $x:=(x^{\prime}%
,x_{N})\in\mathbb{R}^{N-1}\times\mathbb{R}$, $Q^{\prime}(0,r):=(-r,r)^{N-1}$,
and $Q(0,r):=(-r,r)^{N}$. Let $\varphi\in C_{c}^{\infty}(Q^{\prime
}(0,1))\rightarrow\lbrack0,1]$ be such that $\int_{Q^{\prime}(0,1)}%
\varphi(y^{\prime})\,dy^{\prime}=1$. For $0<r\leq r_{0}$, define%
\[
\varphi_{r}(x^{\prime}):=\varphi(x^{\prime}/r).
\]
Consider the function $u_{0}:\Omega\rightarrow\mathbb{R}$ given by%
\[
u_{0}(x):=\left\{
\begin{array}
[c]{ll}%
a & \text{if }x\in Q(0,r)\cap\Omega\text{ and }x_{N}\leq f(x^{\prime}%
)+r^{3}\varphi_{r}(x^{\prime}),\\
b & \text{elsewhere in }\Omega\text{.}%
\end{array}
\right.
\]
Define%
\begin{align*}
\Gamma_{f}  &  :=\{(x^{\prime},f(x^{\prime})):\,x^{\prime}\in Q^{\prime
}(0,r)\},\\
\Gamma_{f+r^{3}\varphi_{r}}  &  :=\{(x^{\prime},f(x^{\prime})+r^{3}\varphi
_{r}(x^{\prime})):\,x^{\prime}\in Q^{\prime}(0,r)\},
\end{align*}
 By contradiction, assume  that $b$ is a minimizer of $\mathcal{F}^{(1)}$. Then
\begin{align*}
\mathcal{F}^{(1)}(b)  &  =\int_{\partial\Omega}\operatorname*{d}%
\nolimits_{W}(b,g)\,d\mathcal{H}^{N-1}\leq\mathcal{F}^{(1)}(u_{0}%
)=C_{W}\mathcal{H}^{N-1}(\Omega\cap\Gamma_{f+r^{3}\varphi_{r}})\\[1pt]
&  \quad+\int_{(\partial\Omega\setminus Q(0,r))\cup(\partial\Omega\cap
\Gamma_{f+r^{3}\varphi_{r}})}\operatorname*{d}\nolimits_{W}(b,g)\,d\mathcal{H}%
^{N-1}\\
&  \quad+\int_{(\partial\Omega\cap Q(0,r))\setminus\Gamma_{f+r^{3}\varphi_{r}%
}}\operatorname*{d}\nolimits_{W}(a,g)\,d\mathcal{H}^{N-1},
\end{align*}
which is equivalent to writing %
\begin{equation}
\int_{(\partial\Omega\cap Q(0,r))\setminus\Gamma_{f+r^{3}\varphi_{r}}%
}(\operatorname*{d}\nolimits_{W}(b,g)-\operatorname*{d}\nolimits_{W}%
(a,g))\,d\mathcal{H}^{N-1}\leq C_{W}\mathcal{H}^{N-1}(\Omega\cap
\Gamma_{f+r^{3}\varphi_{r}}). \label{333}%
\end{equation}
Define $\bar{g}(x^{\prime}):=g(x^{\prime},f(x^{\prime}))$, $x^{\prime}%
\in\mathbb{R}^{N-1}$. Since $\bar{g}(0)=a$ and $a\leq\bar{g}(x^{\prime})$ for
all $x^{\prime}$ small, $0$ is a point of local minimum, and so $\nabla
^{\prime}\bar{g}(0)=0$. Since $W^{1/2}(\rho)\sim(\rho-a)$ as $\rho\rightarrow
a$, by Taylor's formula applied to the function $x^{\prime}\mapsto\int%
_{a}^{\bar{g}(x^{\prime})}W^{1/2}(\rho)\,d\rho$, we can write
\begin{align*}
\operatorname*{d}\nolimits_{W}(b,\bar{g}(x^{\prime}))-\operatorname*{d}%
\nolimits_{W}(a,\bar{g}(x^{\prime}))  &  =2\int_{\bar{g}(x^{\prime})}%
^{b}W^{1/2}(\rho)\,d\rho-2\int_{a}^{\bar{g}(x^{\prime})}W^{1/2}(\rho)\,d\rho\\
&  =C_{W}-4\int_{a}^{\bar{g}(x^{\prime})}W^{1/2}(\rho)\,d\rho\\
&  =C_{W}+O(|x^{\prime}|^{4}).
\end{align*}
Then (\ref{332}) and (\ref{333}) imply
\begin{align*}
&  \int_{Q^{\prime}(0,r)\cap\{\varphi_{r}>0\}}(C_{W}+O(|x^{\prime}%
|^{4}))(1+|\nabla^{\prime}f(x^{\prime})|^{2})^{1/2}dx^{\prime}\\
&  \leq C_{W}\int_{Q^{\prime}(0,r)\cap\{\varphi_{r}>0\}}(1+|\nabla^{\prime
}f(x^{\prime})+r^{3}\nabla^{\prime}\varphi_{r}(x^{\prime})|^{2})^{1/2}%
dx^{\prime},
\end{align*}
or, equivalently,
\begin{align}
&  C_{W}\int_{Q^{\prime}(0,r)\cap\{\varphi_{r}>0\}}\left(  (1+|\nabla^{\prime
}f|^{2})^{1/2}-(1+|\nabla^{\prime}f+r^{3}\nabla^{\prime}\varphi_{r}%
|^{2})^{1/2}\right)  \,dx^{\prime}\label{334}\\
&  \leq Cr^{4}\int_{Q^{\prime}(0,r)\cap\{\varphi_{r}>0\}}(1+|\nabla^{\prime
}f|^{2})^{1/2}dx^{\prime}.\nonumber
\end{align}
Using the fact that $(1+t)^{1/2}\leq1+\frac{1}{2}t$ for $t\geq-1$,we have%
\begin{align*}
(1+|\nabla^{\prime}f+r^{3}\nabla^{\prime}\varphi_{r}|^{2})^{1/2}  &
=(1+|\nabla^{\prime}f|^{2})^{1/2}\left(  1+\frac{2r^{3}\nabla^{\prime}%
f\cdot\nabla^{\prime}\varphi_{r}}{1+|\nabla^{\prime}f|^{2}}+r^{6}\frac
{|\nabla^{\prime}\varphi_{r}|^{2}}{1+|\nabla^{\prime}f|^{2}}\right)  ^{1/2}\\
&  \leq(1+|\nabla^{\prime}f|^{2})^{1/2}+\frac{r^{3}\nabla^{\prime}f\cdot
\nabla^{\prime}\varphi_{r}}{(1+|\nabla^{\prime}f|^{2})^{1/2}}+r^{6}%
\frac{|\nabla^{\prime}\varphi_{r}|^{2}}{(1+|\nabla^{\prime}f|^{2})^{1/2}}.
\end{align*}
Hence,%
\begin{align*}
-C_{W}r^{3}  &  \int_{Q^{\prime}(0,r)\cap\{\varphi_{r}>0\}}\frac
{\nabla^{\prime}f\cdot\nabla^{\prime}\varphi_{r}}{(1+|\nabla^{\prime}%
f|^{2})^{1/2}}dx^{\prime}\\
&  \leq r^{6}C_{W}\int_{Q^{\prime}(0,r)\cap\{\varphi_{r}>0\}}\frac
{|\nabla^{\prime}\varphi_{r}|^{2}}{(1+|\nabla^{\prime}f|^{2})^{1/2}}%
dx^{\prime}+Cr^{4}\int_{Q^{\prime}(0,r)\cap\{\varphi_{r}>0\}}(1+|\nabla
^{\prime}f|^{2})^{1/2}dx^{\prime}.
\end{align*}
Integrating by parts the first integral and using the fact that $\Vert
\nabla^{\prime}\varphi_{r}\Vert_{\infty}\leq\frac{C}{r}$ gives%
\[
C_{W}r^{3}\int_{Q^{\prime}(0,r)}\operatorname{div}_{x^{\prime}}\left(
\frac{\nabla^{\prime}f}{(1+|\nabla^{\prime}f|^{2})^{1/2}}\right)  \varphi
_{r}\,dx^{\prime}\leq Cr^{N+3}%
\]
Dividing this inequality by $r^{N+2}$, and considering the change of variables
$y^{\prime}:=r^{-1}x^{\prime}$, gives%
\[
C_{W}\int_{Q^{\prime}(0,1)}\kappa(ry^{\prime})\varphi(y^{\prime})\,dy^{\prime
}\leq Cr.
\]
Letting $r\rightarrow0^{+}$ and recalling that $\int_{Q^{\prime}(0,1)}%
\varphi(y^{\prime})\,dy^{\prime}=1$, we have%
\[
C_{W}\kappa(0)\leq0,
\]
which is a contradiction.\hfill
\end{proof}

For the proof of the following theorem, we refer to
\cite[Theorem 4.9]{fonseca-kreutz-leoni2025I}.

\begin{theorem}
\label{theorem boundary estimates}Let $\Omega\subset\mathbb{R}^{N}$ be an
open, bounded, connected set with boundary of class $C^{2,d}\ $, $0<d\leq1$.
Assume that $W$ satisfies hypotheses
\eqref{W_Smooth}-\eqref{W' three zeroes} and that $g_{\varepsilon}$ satisfy
hypotheses \eqref{bounds g a=alpha},
\eqref{g epsilon smooth a}-\eqref{g epsilon -g bound a}. Suppose also that
\eqref{u0=b} holds. Let $0<\delta<<1$, then there exist $\mu>0$ and $C>0$,
independent of $\varepsilon$ and $\delta$, such that for all $\varepsilon$
sufficiently small the following estimate holds%
\begin{equation}
0\leq b-u_{\varepsilon}(x)\leq Ce^{-\mu\delta/\varepsilon}\quad\text{for }%
x\in\Omega\setminus\Omega_{2\delta}. \label{decay b}%
\end{equation}

\end{theorem}

\section{Second-Order $\Gamma$-Limit}

\label{section main theorems}In this section, we finally prove Theorem
\ref{theorem main}.

\begin{theorem}
[Second-Order $\Gamma$-Limsup]\label{theorem limsup}Let $\Omega\subset\mathbb{R}^{N}$
be an open, bounded, connected set with boundary of class $C^{2,d}\ $,
$0<d\leq1$. Assume that $W$ satisfies
\eqref{W_Smooth}-\eqref{W' three zeroes} and that $g_{\varepsilon}$ satisfy \eqref{bounds g a=alpha},
\eqref{g epsilon smooth a}-\eqref{g epsilon -g bound a}. Suppose also that
\eqref{u0=b} holds. Then there exists $\{u_{\varepsilon}\}_{\varepsilon}$ in
$H^{1}(\Omega)$ such that $\operatorname*{tr}u_{\varepsilon}=g_{\varepsilon}$
on $\partial\Omega$, $u_{\varepsilon}\rightarrow b$ in $L^{1}(\Omega)$, and%
\[
\limsup_{\varepsilon\rightarrow0^{+}}\mathcal{F}_{\varepsilon}^{(2)}%
(u_{\varepsilon})\leq\frac{C_{W}}{2^{1/2}(W^{\prime\prime}(a))^{1/2}}%
\int_{\partial\Omega\cap\{g=a\}}\kappa(y)\,d\mathcal{H}^{N-1}(y).
\]
Here, $\mathcal{F}^{(2)}$ is defined in \eqref{F 2 epsilon a=aplha}, $\kappa$
is the mean curvature of $\partial\Omega$, and $C_{W}$ is the constant defined
in \eqref{cW definition}.
\end{theorem}

\begin{proof}
By Lemma \ref{lemma diffeomorphism}, for $\delta>0$ sufficiently small, the
function $\Phi:\partial\Omega\times\lbrack0,\delta]\rightarrow\overline
{\Omega}_{\delta}$ is of class $C^{1\,d}$. In turn, the function%
\[
\omega(y,t):=\det J_{\Phi}(y,t)
\]
is of class $C^{1,d}$,
\begin{equation}
\omega_{1}:=\min_{y\in\partial\Omega}\omega(y,0)>0,\quad\omega(y,0)=1\quad
\text{for all }y\in\partial\Omega, \label{omega prioperties}%
\end{equation}
and%
\begin{equation}
\frac{\partial\omega}{\partial t}(y,0)=\kappa(y)\quad\text{for all }%
y\in\partial\Omega, \label{curvature 1}%
\end{equation}
where $\kappa(y)$ is the mean curvature of $\partial\Omega$ at $y$.

In view of (\ref{set g=a}), $\operatorname*{dist}(\{g=a\},\partial
\{\kappa<0\})=:\rho_{0}>0$. Let
\begin{align*}
K_{1}  &  :=\{x\in\partial\Omega:\,\operatorname*{dist}(x,\{g=a\})\geq\rho
_{0}/2\},\\
K_{2}  &  :=\{x\in\partial\Omega:\,\operatorname*{dist}(x,\{g=a\})\leq\rho
_{0}/2\}.
\end{align*}
Then
\[
\min_{K_{1}}g=:g_{-}>a.
\]
Fix
\begin{equation}
0<\omega_{0}<\frac{1}{4}\frac{C_{W}-\operatorname*{d}\nolimits_{W}(a,g_{-}%
)}{C_{W}}\omega_{1}. \label{omega0}%
\end{equation}
By taking $\delta>0$ sufficiently small, we can assume that%
\begin{equation}
|\omega(y,t_{1})-\omega(y,t_{2})|\leq\omega_{0} \label{omega uc}%
\end{equation}
for all $y\in\partial\Omega$ and all $t_{1},t_{2}\in\lbrack0,\delta]$. Since
$\omega$ is of class $C^{1,d}$ and $\kappa<0$ in $K_{2}$, by
(\ref{curvature 1}) and taking $\delta$ even smaller, we can assume that%
\begin{equation}
\frac{\partial\omega}{\partial t}(y,t)<0 \label{omega decreasing}%
\end{equation}
for all $y\in K_{2}$ and $t\in\lbrack0,\delta]$.

For each $y\in\overline{\Omega}$, define
\begin{equation}
\Psi_{\varepsilon}(y,r):=\int_{g_{\varepsilon}(y)}^{r}\frac{\varepsilon
}{(\varepsilon+W(s))^{1/2}}\,ds, \label{Psi epsilon}%
\end{equation}
and
\begin{equation}
0<T_{\varepsilon}(y):=\Psi_{\varepsilon}(y,b). \label{T epsilon y def}%
\end{equation}
Note that $T_{\varepsilon}\in C^{1}(\overline{\Omega})$,  with
\begin{equation}
T_{\varepsilon}(y)\leq\int_{a}^{b}\frac{\varepsilon}{(\varepsilon+W(s))^{1/2}%
}\,ds\leq C_{0}\varepsilon|\log\varepsilon| \label{T epsilon y}%
\end{equation}
for all $0<\varepsilon<\varepsilon_{0}$ and all $y\in\partial\Omega$ by
(\ref{near b log}), where $C_{0}>0$ and $\varepsilon_{0}>0$ depend only on $W$

For each fixed $y\in\partial\Omega$, let $v_{\varepsilon}(y,\cdot
):[0,T_{\varepsilon}(y)]\rightarrow\lbrack g_{\varepsilon}(y),b]$ be the
inverse of $\Psi_{\varepsilon}(y,\cdot)$. Then $v_{\varepsilon}\left(
y,0\right)  =g_{\varepsilon}(y)$, $v_{\varepsilon}(y,T_{\varepsilon}(y))=b$,
and
\begin{equation}
\frac{\partial v_{\varepsilon}}{\partial t}(y,t)=\frac{(\varepsilon
+W(v_{\varepsilon}\left(  y,t\right)  ))^{1/2}}{\varepsilon}
\label{partial v epsilon t}%
\end{equation}
for $t\in\lbrack0,T_{\varepsilon}(y)]$. Assume first that $g_{\varepsilon}\in
C^{1}(\partial\Omega)$. Then,  by standard results on the smooth dependence of
solutions on a parameter (see, e.g. \cite[Section 2.4]{gerald-book2012}), we
see that $v_{\varepsilon}$ is of class $C^{1}$ in the variables $(y,t)$.
Extend $v_{\varepsilon}(y,t)$ to be equal to $b$ for $t>T_{\varepsilon}(y)$.

We have
\[
v_{\varepsilon}(y,\Psi_{\varepsilon}(y,r))=r
\]
for all $g_{\varepsilon}(y)\leq r\leq b$. For every $y\in\partial\Omega$ and
every tangent vector $\tau$ to $\partial\Omega$ at $y$, differentiating in the
direction $\tau$ gives
\[
\frac{\partial v_{\varepsilon}}{\partial\tau}(y,\Psi_{\varepsilon}%
(y,r))+\frac{\partial v_{\varepsilon}}{\partial t}(y,\Psi_{\varepsilon
}(y,r))\frac{\partial\Psi_{\varepsilon}}{\partial\tau}(y,r)=0.
\]
Hence,
\[
\frac{\partial v_{\varepsilon}}{\partial\tau}(y,t)+\frac{\partial
v_{\varepsilon}}{\partial t}(y,t)\frac{\partial\Psi_{\varepsilon}}%
{\partial\tau}(y,r)=0
\]
for all $y\in\partial\Omega$ and $t\in\lbrack0,T_{\varepsilon}(y))$.

By (\ref{Psi epsilon}),
\[
\frac{\partial\Psi_{\varepsilon}}{\partial\tau}(y,r)=-\frac{\varepsilon
}{(\varepsilon+W(g_{\varepsilon}(y)))^{1/2}}\frac{\partial g_{\varepsilon}%
}{\partial\tau}(y),
\]
and so by (\ref{partial v epsilon t}), we have%
\begin{align*}
\frac{\partial v_{\varepsilon}}{\partial\tau}(y,t)  &  =-\frac{\partial
v_{\varepsilon}}{\partial t}(y,t)\frac{\partial\Psi_{\varepsilon}}%
{\partial\tau}(y,r)\\
&  =\frac{(\varepsilon+W(v_{\varepsilon}\left(  y,t\right)  ))^{1/2}%
}{(\varepsilon+W(g_{\varepsilon}(y)))^{1/2}}\frac{\partial g_{\varepsilon}%
}{\partial\tau}(y)
\end{align*}
for $t\in\lbrack0,T_{\varepsilon}(y))$, while $\frac{\partial v_{\varepsilon}%
}{\partial\tau}(y,t)=0$ for $t>T_{\varepsilon}(y)$. Observe that if
$g_{\varepsilon}(y)\geq c$, then since $W$ is decreasing for $c\leq s\leq s$
and $v_{\varepsilon}(y,\cdot)$ is increasing, we have $W(v_{\varepsilon
}\left(  y,t\right)  )\leq W(g_{\varepsilon}\left(  y\right)  )$.  Thus,
$\left\vert \frac{\partial v_{\varepsilon}}{\partial\tau}(y,t)\right\vert
\leq\left\vert \frac{\partial g_{\varepsilon}}{\partial\tau}(y)\right\vert $.
On the other hand, if $g_{\varepsilon}(y)\leq c$, then by (\ref{bounds g}),
\[
(\varepsilon+W(g_{\varepsilon}(y)))^{1/2}\geq\min_{\lbrack g_{-},c]}%
W^{1/2}=:W_{0}>0.
\]
Since $a\leq v_{\varepsilon}(y,t)\leq b$, in both cases, we have%
\begin{equation}
\left\vert \frac{\partial v_{\varepsilon}}{\partial\tau}(y,t)\right\vert
\leq\left\{
\begin{array}
[c]{ll}%
C\left\vert \frac{\partial g_{\varepsilon}}{\partial\tau}(y)\right\vert  &
\text{if }y\in\partial\Omega\text{ and }t\in\lbrack0,T_{\varepsilon}(y)),\\
0 & \text{if }y\in\partial\Omega\text{ and }t\in(T_{\varepsilon}(y),\delta].
\end{array}
\right.  \label{partial v epsilon y}%
\end{equation}
If $g_{\varepsilon}\in H^{1}(\partial\Omega)$, a density argument shows that
$v_{\varepsilon}\in H^{1}(\partial\Omega\times(0,\delta))$ and that
(\ref{partial v epsilon t}) and (\ref{partial v epsilon y}) continues to hold a.e.

Set%
\begin{equation}
u_{\varepsilon}(x):=\left\{
\begin{array}
[c]{ll}%
v_{\varepsilon}(\Phi^{-1}(x)) & \text{if }x\in\Omega_{\delta},\\
b & \text{if }x\in\Omega\setminus\Omega_{\delta},
\end{array}
\right.  . \label{u epsilon bl}%
\end{equation}
Then $u_{\varepsilon}\in H^{1}(\Omega)$, with
\begin{equation}
|\nabla u_{\varepsilon}(x)|^{2}\leq\left\vert \frac{\partial v_{\varepsilon}%
}{\partial t}(\Phi^{-1}(x))\right\vert ^{2}+C\Vert\nabla y\Vert_{L^{\infty
}(\Omega_{\delta})}^{2}\left\vert \nabla_{\tau}v_{\varepsilon}(\Phi
^{-1}(x))\right\vert ^{2}, \label{299b}%
\end{equation}
where we used the facts that $\Phi^{-1}(x)=(y(x),\operatorname*{dist}%
(x,\partial\Omega))$, $\left\vert \nabla\operatorname*{dist}(x,\partial
\Omega)\right\vert =1$, and $\tau\cdot\nabla\operatorname*{dist}%
(x,\partial\Omega)=0$ for every vector $\tau$ such that $\tau\cdot\nu(y)=0$.

In view of Lemma \ref{lemma diffeomorphism}, we can use the change of
variables $x:=\Phi(y,t)$ and Tonelli's theorem to write%
\begin{align}
\mathcal{F}_{\varepsilon}^{(2)}(u_{\varepsilon})  &  =\frac{1}{\varepsilon
|\log\varepsilon|}\int_{\partial\Omega}\int_{0}^{\delta}\left(  \frac
{1}{\varepsilon}W(u_{\varepsilon}(\Phi(y,t)))+\varepsilon|\nabla
u_{\varepsilon}(\Phi(y,t))|^{2}\right)  \omega(y,t)\,dtd\mathcal{H}%
^{N-1}(y)\nonumber\\
&  \quad-\frac{1}{\varepsilon|\log\varepsilon|}\int_{\partial\Omega
}\operatorname*{d}\nolimits_{W}(g(y),b)\,d\mathcal{H}^{N-1}%
(y)\label{F 2 epsilon}\\
&  \leq\frac{1}{\varepsilon|\log\varepsilon|}\left(  \int_{\partial\Omega}%
\int_{0}^{\delta}\left(  \frac{1}{\varepsilon}W(v_{\varepsilon}%
(y,t))+\varepsilon\left\vert \frac{\partial v_{\varepsilon}}{\partial
t}(y,t)\right\vert ^{2}\right)  \omega(y,t)\,dtd\mathcal{H}^{N-1}(y)\right.
\nonumber\\
&  \quad\left.  -\frac{1}{\varepsilon|\log\varepsilon|}\int_{\partial\Omega
}\operatorname*{d}\nolimits_{W}(g(y),b)\,d\mathcal{H}^{N-1}(y)\right)
\nonumber\\
&  \quad+\frac{C}{|\log\varepsilon|}\Vert\nabla y\Vert_{L^{\infty}%
(\Omega_{\delta})}^{2}\int_{\partial\Omega}\int_{0}^{\delta}\left\vert
\nabla_{\tau}v_{\varepsilon}(y,t)\right\vert ^{2}\omega(y,t)\,dtd\mathcal{H}%
^{N-1}(y)=:\mathcal{A}+\mathcal{B}.\nonumber
\end{align}
To estimate $\mathcal{A}$, we consider two cases.

\noindent \textbf{Case 1: $g(y)=a$.}  Fix $0<\eta<\frac{1}{4}$,  and let
$y\in\partial\Omega$ be such $g(y)=a$, then by (\ref{omega decreasing}),
$\frac{\partial\omega}{\partial t}(y,t)<0$ for all $t\in\lbrack0,\delta]$.
Thus, also by (\ref{omega prioperties}), we can apply Theorem
\ref{theorem 1d limsup alpha=a} to obtain%
\begin{align*}
&  \frac{1}{\varepsilon|\log\varepsilon|}\int_{0}^{\delta}\left(  \frac
{1}{\varepsilon}W(v_{\varepsilon}(y,t))+\varepsilon\left\vert \frac{\partial
v_{\varepsilon}}{\partial t}(y,t)\right\vert ^{2}\right)  \omega
(y,t)\,dt-\frac{C_{W}}{\varepsilon|\log\varepsilon|}\\
&  \leq(1+\eta)\frac{C_{W}}{2^{1/2}(W^{\prime\prime}(a))^{1/2}}\frac
{\partial\omega}{\partial t}(y,0)+\frac{C}{|\log\varepsilon|}%
\end{align*}
for all $0<\varepsilon<\varepsilon_{\eta}$, for some $0<\varepsilon_{\eta}<1$
depending on $\eta$, $A_{0}$, $B_{0}$, $\delta$, $\omega$, and $W$, and for
some constant $C>0$, depending on $A_{0}$, $B_{0}$, $T$, $\delta$, and $W$.
Integrating over the set $\{g=a\}$ gives%
\begin{align*}
\mathcal{A}_{1}:=\frac{1}{\varepsilon|\log\varepsilon|}  &  \int%
_{\partial\Omega\cap\{g=a\}}\int_{0}^{\delta}\left(  \frac{1}{\varepsilon
}W(v_{\varepsilon}(y,t))+\varepsilon\left\vert \frac{\partial v_{\varepsilon}%
}{\partial t}(y,t)\right\vert ^{2}\right)  \omega(y,t)\,dtd\mathcal{H}%
^{N-1}(y)\\
&  \quad-\frac{1}{\varepsilon|\log\varepsilon|}\mathcal{H}^{N-1}%
(\partial\Omega\cap\{g=a\})\\
&  \leq(1+\eta)\frac{C_{W}}{2^{1/2}(W^{\prime\prime}(a))^{1/2}}\int%
_{\partial\Omega\cap\{g=a\}}\kappa(y)\,d\mathcal{H}^{N-1}(y)\\
&  \quad+\frac{C}{|\log\varepsilon|}\mathcal{H}^{N-1}(\partial\Omega
\cap\{g=a\}).
\end{align*}
Letting $\varepsilon\rightarrow0^{+}$ gives%
\[
\limsup_{\varepsilon\rightarrow0^{+}}\mathcal{A}_{1}\leq(1+\eta)\frac{C_{W}%
}{2^{1/2}(W^{\prime\prime}(a))^{1/2}}\int_{\partial\Omega\cap\{g=a\}}%
\kappa(y)\,d\mathcal{H}^{N-1}(y).
\]

\noindent \textbf{Case 2: $g(y)>a$.} Now let $y\in\partial\Omega$ be such that $g(y)>a$. If $y\in K_{1}$, then $\omega(y,\cdot)$ satisfies
(\ref{eta 0 alpha-}) by (\ref{omega0}) and (\ref{omega uc}), while if $y\in
K_{2}$, then $\omega(y,\cdot)$ is strictly increasing in $[0,\delta]$ by
(\ref{omega decreasing}), and so it satisfies (\ref{eta 0 alpha-}) with
$\omega_{0}=0$. Thus, in both cases, also by (\ref{omega prioperties}),
given $l>0$, we can apply Remark \ref{remark limsup a<alpha} to get%
\begin{align*}
&  \frac{1}{\varepsilon|\log\varepsilon|}\int_{0}^{\delta}\left(  \frac
{1}{\varepsilon}W(v_{\varepsilon}(y,t))+\varepsilon\left\vert \frac{\partial
v_{\varepsilon}}{\partial t}(y,t)\right\vert ^{2}\right)  \omega
(y,t)\,dt-\frac{\operatorname*{d}\nolimits_{W}(g(y),b)}{\varepsilon
|\log\varepsilon|}\\
&  \leq\frac{C}{|\log\varepsilon|}+\frac{1}{|\log\varepsilon|}\int_{0}%
^{l}2W(p_{\varepsilon}(y,s))s\,ds\frac{\partial\omega}{\partial t}%
(y,0)+Ce^{-2\sigma l}\left(  2\sigma l+1\right)  \frac{1}{|\log\varepsilon|}\\
&  \quad+\frac{C\varepsilon^{2\gamma}l}{|\log\varepsilon|}+C\varepsilon
|\log\varepsilon|+C\varepsilon^{d}|\log\varepsilon|^{d}+C\frac{\varepsilon
^{2\gamma-2}}{|\log\varepsilon|}%
\end{align*}
for all $0<\varepsilon<\varepsilon_{0}$, where $p_{\varepsilon}%
(y,s):=v_{\varepsilon}(y,\varepsilon s)$ and the constants $C$ and
$\varepsilon_{0}>0$ depend on $A_{0}$, $B_{0}$, $\delta$, $\omega$, and $W$.
Since $a\leq p_{\varepsilon}\leq b$,
\[
\int_{0}^{l}2W(p_{\varepsilon}(y,s))s\,ds\leq l^{2}\max_{[a,b]}W,
\]
by integrating over $\partial\Omega\setminus\{g=a\}$ and taking $\varepsilon
_{0}$ smaller if necessary (depending on $l$), we obtain%
\begin{align*}
\mathcal{A}_{2}:=\frac{1}{\varepsilon|\log\varepsilon|}  &  \int%
_{\partial\Omega\setminus\{g=a\}}\int_{0}^{\delta}\left(  \frac{1}%
{\varepsilon}W(v_{\varepsilon}(y,t))+\varepsilon\left\vert \frac{\partial
v_{\varepsilon}}{\partial t}(y,t)\right\vert ^{2}\right)  \omega
(y,t)\,dtd\mathcal{H}^{N-1}(y)\\
&  \quad-\frac{1}{\varepsilon|\log\varepsilon|}\int_{\partial\Omega
\setminus\{g=a\}}\operatorname*{d}\nolimits_{W}(g(y),b)\,d\mathcal{H}%
^{N-1}(y)\\
&  \leq\frac{C}{|\log\varepsilon|}\mathcal{H}^{N-1}(\partial\Omega
\setminus\{g=a\})
\end{align*}
for all $0<\varepsilon_{0}<1$. Letting $\varepsilon\rightarrow0^{+}$ gives%
\[
\limsup_{\varepsilon\rightarrow0^{+}}\mathcal{A}_{2}\leq0.
\]
In conclusion, we have shown that
\begin{equation}
\limsup_{\varepsilon\rightarrow0^{+}}\mathcal{A}\leq(1+\eta)\frac{C_{W}%
}{2^{1/2}(W^{\prime\prime}(a))^{1/2}}\int_{\partial\Omega\cap\{g=a\}}%
\kappa(y)\,d\mathcal{H}^{N-1}(y) \label{299bb}%
\end{equation}
for every $0<\eta<1$. We now let $\eta\rightarrow0^{+}$.

On the other hand, by \eqref{g epsilon bound derivatives a}, (\ref{T epsilon y}), and (\ref{partial v epsilon y}),
\begin{align}
\mathcal{B}  &  \leq\frac{C}{|\log\varepsilon|}\Vert\nabla y\Vert_{L^{\infty
}(\Omega_{\delta})}^{2}\int_{\partial\Omega}\left\vert \nabla_{\tau
}g_{\varepsilon}(y)\right\vert ^{2}\int_{0}^{T_{\varepsilon}(y)}%
\omega(y,t)\,dtd\mathcal{H}^{N-1}(y)\nonumber\\
&  \leq C\varepsilon\Vert\omega\Vert_{L^{\infty}(\partial\Omega\times
\lbrack0,\delta])}\int_{\partial\Omega}\left\vert \partial_{\tau
}g_{\varepsilon}(y)\right\vert ^{2}d\mathcal{H}^{N-1}(y)=o(1).\label{299c}%
\end{align}

By (\ref{F 2 epsilon}), (\ref{299bb}), (\ref{299c}), we have%
\begin{equation}
\limsup_{\varepsilon\rightarrow0^{+}}\mathcal{F}_{\varepsilon}^{(2)}%
(u_{\varepsilon})\leq\frac{C_{W}}{2^{1/2}(W^{\prime\prime}(a))^{1/2}}%
\int_{\partial\Omega\cap\{g=a\}}\kappa(y)\,d\mathcal{H}^{N-1}(y).\nonumber
\end{equation}
\textbf{Step 2: }We claim that
\[
u_{\varepsilon}\rightarrow b\quad\text{in }L^{1}(\Omega).
\]
In view of Lemma \ref{lemma diffeomorphism}, we can use the change of
variables $x:=\Phi(y,t)$ and Tonelli's theorem to write%
\begin{align*}
\int_{\Omega}|u_{\varepsilon}-b|\,dx  &  =\int_{\partial\Omega}\int%
_{0}^{\delta}|u_{\varepsilon}(\Phi(y,t)))-b|\omega(y,t)\,dtd\mathcal{H}%
^{N-1}(y)\\
&  =\int_{\partial\Omega}\int_{0}^{T_{\varepsilon}(y)}|v_{\varepsilon
}(y,t)-b|\omega(y,t)\,dtd\mathcal{H}^{N-1}(y)\\
&  \leq C\varepsilon|\log\varepsilon|,
\end{align*}
where we used the fact that $v_{\varepsilon}(y,t)=b$ for $t\geq T_{\varepsilon
}(y)$ and (\ref{T epsilon y}).\hfill
\end{proof}

In the next proof, we use the localized energy
\[
E_{\varepsilon}(u,E):=\frac{1}{\varepsilon|\log\varepsilon|}\int_{E}\left(
\frac{1}{\varepsilon}W(u)+\varepsilon|\nabla u|^{2}\right)  \,dx,\quad u\in
H^{1}(\Omega),
\]
defined for measurable sets $E\subseteq\Omega$.

\begin{theorem}
[Second-Order $\Gamma$-Liminf]\label{theorem liminf}Let $\Omega\subset\mathbb{R}^{N}$
be an open, bounded, connected set with a boundary of class $C^{2,d}\ $,
$0<d\leq1$. Assume that $W$ satisfies
\eqref{W_Smooth}-\eqref{W' three zeroes} and that $g_{\varepsilon}$ satisfy \eqref{bounds g a=alpha},
\eqref{g epsilon smooth a}-\eqref{g epsilon -g bound a}. Suppose also that
\eqref{u0=b} holds. Then%
\[
\liminf_{\varepsilon\rightarrow0^{+}}\mathcal{F}_{\varepsilon}^{(2)}%
(u_{\varepsilon})\geq\frac{C_{W}}{2^{1/2}(W^{\prime\prime}(a))^{1/2}}%
\int_{\partial\Omega\cap\{g=a\}}\kappa(y)\,d\mathcal{H}^{N-1}(y).
\]

\end{theorem}

\begin{proof}
We define $\omega$ and $\delta>0$ as in the first part of the proof of Theorem
\ref{theorem limsup}.

By Theorem \ref{theorem boundary estimates} (with $\Omega_{\delta}$ and
$\Omega_{2\delta}$ replaced by $\Omega_{\delta/2}$ and $\Omega_{\delta}$,
respectively), we can assume that
\begin{equation}
0\leq b-u_{\varepsilon}(x)\leq Ce^{-\mu\delta/\varepsilon}\quad\text{for }%
x\in\Omega\setminus\Omega_{\delta} \label{902}%
\end{equation}
for all $0<\varepsilon<\varepsilon_{\delta}$.

Write
\begin{align*}
\mathcal{F}_{\varepsilon}^{(2)}(u_{\varepsilon})  &  =E_{\varepsilon
}(u_{\varepsilon},\Omega\setminus\Omega_{\delta})\\
&  \quad+\left(  E_{\varepsilon}(u_{\varepsilon},\Omega_{\delta})-\frac
{1}{\varepsilon|\log\varepsilon|}\int_{\partial\Omega}\operatorname*{d}%
\nolimits_{W}(g,b)\,d\mathcal{H}^{N-1}\right) \\
&  \quad=:\mathcal{A}+\mathcal{B}.
\end{align*}
Since $\mathcal{A}\geq0$, it remains to evaluate $\mathcal{B}$. In view
of Lemma \ref{lemma diffeomorphism}, we can use the change of variables
$x:=\Phi(y,t)$ and Tonelli's theorem to write%
\[
E_{\varepsilon}(u_{\varepsilon},\Omega_{\delta})=\frac{1}{\varepsilon
|\log\varepsilon|}\int_{\partial\Omega}\int_{0}^{\delta}\left(  \frac
{1}{\varepsilon}W(u_{\varepsilon}(\Phi(y,t)))+\varepsilon|\nabla
u_{\varepsilon}(\Phi(y,t))|^{2}\right)  \omega(y,t)\,dtd\mathcal{H}^{N-1}(y).
\]
Since $u_{\varepsilon}\in C^{1}(\overline{\Omega})$, if we define%
\[
\tilde{u}_{\varepsilon}(y,t):=u_{\varepsilon}(y+t\nu(y)),
\]
we have that%
\[
\frac{\partial\tilde{u}_{\varepsilon}}{\partial t}(y,t)=\frac{\partial
u_{\varepsilon}}{\partial\nu(y)}(y+t\nu(y)),
\]
and so,%
\begin{align}
\mathcal{B}  &  \geq\int_{\partial\Omega}\left[  \frac{1}{\varepsilon
|\log\varepsilon|}\int_{0}^{\delta}\left(  \frac{1}{\varepsilon}W(\tilde
{u}_{\varepsilon}(y,t))+\varepsilon\left\vert \frac{\partial\tilde
{u}_{\varepsilon}}{\partial t}(y,t)\right\vert ^{2}\right)  \omega
(y,t)\,dt\right. \label{900}\\
&  -\left.  \frac{1}{\varepsilon|\log\varepsilon|}\operatorname*{d}%
\nolimits_{W}(g(y),b)\right]  d\mathcal{H}^{N-1}(y)\nonumber
\end{align}
For $y\in\partial\Omega$, in view of (\ref{902}) we have that%
\begin{equation}
b-C_{\rho}e^{-\mu_{\rho}\delta/(2\varepsilon)}\leq\tilde{u}_{\varepsilon
}(y,\delta)\leq b. \label{903}%
\end{equation}
Let $v_{\varepsilon}^{y}\in H^{1}([0,\delta])$ be the minimizer of the
functional
\[
v\mapsto\int_{0}^{\delta}\left(  \frac{1}{\varepsilon}W(v(t))+\varepsilon
|v^{\prime}(t)|^{2}\right)  \omega(y,t)\,dt
\]
defined for all $v\in H^{1}([0,\delta])$ such that $v(0)=g_{\varepsilon}(y)$
and $v(\delta)=\tilde{u}_{\varepsilon}(y,\delta)$.

There are now two cases. If $y\in\partial\Omega$ is such $g(y)=a$, then by
(\ref{omega decreasing}), $\frac{\partial\omega}{\partial t}(y,t)<0$ for all
$t\in\lbrack0,\delta]$. Thus, also by (\ref{omega prioperties}), given
$0<\eta<\frac{1}{4}$, we can apply Theorem \ref{theorem liminf 1d alpha=a} to
obtain%
\begin{align*}
&  \frac{1}{\varepsilon|\log\varepsilon|}\int_{0}^{\delta}\left(  \frac
{1}{\varepsilon}W(\tilde{u}_{\varepsilon}(y,t))+\varepsilon\left\vert
\frac{\partial\tilde{u}_{\varepsilon}}{\partial t}(y,t)\right\vert
^{2}\right)  \omega(y,t)\,dt-\frac{C_{W}}{\varepsilon|\log\varepsilon|}\\
&  \geq\frac{1}{\varepsilon|\log\varepsilon|}\int_{0}^{\delta}\left(  \frac
{1}{\varepsilon}W(v_{\varepsilon}^{y}(t))+\varepsilon\left\vert
(v_{\varepsilon}^{y})^{\prime}(t)\right\vert ^{2}\right)  \omega
(y,t)\,dt-\frac{C_{W}}{\varepsilon|\log\varepsilon|}\\
&  \geq\left(  1-\eta\right)  \frac{C_{W}}{2^{1/2}(W^{\prime\prime}(a))^{1/2}%
}\frac{\partial\omega}{\partial t}(y,0)-\frac{C_{\eta}}{|\log\varepsilon|}%
\end{align*}
for all $0<\varepsilon<\varepsilon_{\eta}$, for some $0<\varepsilon_{\eta}<1$
and $C_{\eta}>0$ depending on $\eta$, $A_{0}$, $B_{0}$, $\delta$, $\omega$,
and $W$. Integrating over the set $\{g=a\}$ gives%
\begin{align*}
\mathcal{B}_{1}  &  :=\frac{1}{\varepsilon|\log\varepsilon|}\int%
_{\partial\Omega\cap\{g=a\}}\int_{0}^{\delta}\left(  \frac{1}{\varepsilon
}W(\tilde{u}_{\varepsilon}(y,t))+\varepsilon\left\vert \frac{\partial\tilde
{u}_{\varepsilon}}{\partial t}(y,t)\right\vert ^{2}\right)  \omega(y,t)\,dt\\
&  \quad-\frac{1}{\varepsilon|\log\varepsilon|}\mathcal{H}^{N-1}%
(\partial\Omega\cap\{g=a\})\\
&  \geq(1-\eta)\frac{C_{W}}{2^{1/2}(W^{\prime\prime}(a))^{1/2}}\int%
_{\partial\Omega\cap\{g=a\}}\kappa(y)\,d\mathcal{H}^{N-1}(y)\\
&  \quad-\frac{C_{\eta}}{|\log\varepsilon|}\mathcal{H}^{N-1}(\partial
\Omega\cap\{g=a\}).
\end{align*}
On the other hand, if $y\in\partial\Omega$ is such $g(y)>a$, then there are
two cases. If $y\in K_{1}$, then $\omega(y,\cdot)$ satisfies
(\ref{eta 0 alpha-}) by (\ref{omega0}) and (\ref{omega uc}), while if $y\in
K_{2}$, then $\omega(y,\cdot)$ is strictly increasing in $[0,\delta]$ by
(\ref{omega decreasing}), and so it satisfies (\ref{eta 0 alpha-}) with
$\omega_{0}=0$. Thus, in both cases, , also by (\ref{omega prioperties}),
given $l>0$, we can apply Theorem \ref{theorem liminf 1d a<alpha} to find
$0<\varepsilon_{0}<1$, $C>0$, and $l_{0}>1$, depending only on $g_{\pm}$, $k$,
$a$, $b$, $\delta$, $\omega$, and $W$ such that%
\begin{align*}
&  \frac{1}{\varepsilon|\log\varepsilon|}\int_{0}^{\delta}\left(  \frac
{1}{\varepsilon}W(\tilde{u}_{\varepsilon}(y,t))+\varepsilon\left\vert
\frac{\partial\tilde{u}_{\varepsilon}}{\partial t}(y,t)\right\vert
^{2}\right)  \omega(y,t)\,dt-\frac{1}{\varepsilon|\log\varepsilon
|}\operatorname*{d}\nolimits_{W}(b,g(y))\\
&  \geq\frac{1}{\varepsilon|\log\varepsilon|}\int_{0}^{\delta}\left(  \frac
{1}{\varepsilon}W(v_{\varepsilon}^{y}(t))+\varepsilon|(v_{\varepsilon}%
^{y})^{\prime}(t)|^{2}\right)  \omega(y,t)\,dt-\frac{1}{\varepsilon
|\log\varepsilon|}\operatorname*{d}\nolimits_{W}(b,g(y))\\
&  \geq\frac{2}{|\log\varepsilon|}\frac{\partial\omega}{\partial t}%
(y,0)\int_{0}^{l}W^{1/2}(w_{\varepsilon})w_{\varepsilon}^{\prime}%
s\,ds-\frac{Ce^{-l\mu}\left(  l\mu+1\right)  }{|\log\varepsilon|}%
-\frac{Cl\varepsilon^{1/2}}{|\log\varepsilon|}-C\varepsilon^{\gamma_{1}}%
|\log\varepsilon|^{1+\gamma_{0}}%
\end{align*}
for all $0<\varepsilon<\varepsilon_{0}$ and $l>l_{0}$, where $w_{\varepsilon
}(s):=v_{\varepsilon}^{y}(\varepsilon s)$ for $s\in\lbrack0,\delta
\varepsilon^{-1}]$. Since $a\leq w_{\varepsilon}\leq b$, and $\Vert
w_{\varepsilon}^{\prime}\Vert_{\infty}\leq C_{0}$, by Corollary
\ref{corollary bounded derivative},
\[
\int_{0}^{l}W^{1/2}(w_{\varepsilon})|w_{\varepsilon}^{\prime}|s\,ds\leq
l^{2}C_{0}\max_{[a,b]}W^{1/2}.
\]
Hence, by integrating over $\partial\Omega\setminus\{g=a\}$, we obtain%
\begin{align*}
\mathcal{B}_{2}:=\frac{1}{\varepsilon|\log\varepsilon|}  &  \int%
_{\partial\Omega\setminus\{g=a\}}\int_{0}^{\delta}\left(  \frac{1}%
{\varepsilon}W(\tilde{u}_{\varepsilon}(y,t))+\varepsilon\left\vert
\frac{\partial\tilde{u}_{\varepsilon}}{\partial t}(y,t)\right\vert
^{2}\right)  \omega(y,t)\,dtd\mathcal{H}^{N-1}(y)\\
&  \quad-\frac{1}{\varepsilon|\log\varepsilon|}\int_{\partial\Omega
\setminus\{g=a\}}\operatorname*{d}\nolimits_{W}(g(y),b)\,d\mathcal{H}%
^{N-1}(y)\\
&  \geq-\frac{C_{l}}{|\log\varepsilon|}\mathcal{H}^{N-1}(\partial
\Omega\setminus\{g=a\}).
\end{align*}
By combining the estimates for $\mathcal{B}_{1}$ and $\mathcal{B}_{2}$, we
have%
\begin{align*}
\mathcal{F}_{\varepsilon}^{(2)}(u_{\varepsilon})  &  \geq\mathcal{B}%
\geq(1-\eta)\frac{C_{W}}{2^{1/2}(W^{\prime\prime}(a))^{1/2}}\int%
_{\partial\Omega\cap\{g=a\}}\kappa(y)\,d\mathcal{H}^{N-1}(y)\\
&  \quad-\frac{C_{\eta}}{|\log\varepsilon|}\mathcal{H}^{N-1}(\partial
\Omega\cap\{g=a\})-\frac{C_{l}}{|\log\varepsilon|}\mathcal{H}^{N-1}%
(\partial\Omega\setminus\{g=a\}).
\end{align*}
In turn,%
\[
\liminf_{\varepsilon\rightarrow0^{+}}\mathcal{F}_{\varepsilon}^{(2)}%
(u_{\varepsilon})\geq(1-\eta)\frac{C_{W}}{2^{1/2}(W^{\prime\prime}(a))^{1/2}%
}\int_{\partial\Omega\cap\{g=a\}}\kappa(y)\,d\mathcal{H}^{N-1}(y).
\]
We conclude by letting $\eta\rightarrow0^{+}$. \hfill
\end{proof}

\section*{Acknowledgements}

The research of I. Fonseca was partially supported by the National Science
Foundation under grants Nos. DMS-2205627 and DMS-2108784 and DMS-23423490 and G. Leoni
under grant No. DMS-2108784.

G. Leoni thanks R. Murray and I. Tice for useful conversations on
the subject of this paper.

\bibliographystyle{abbrv}
\bibliography{modica-mortola-refs}

\end{document}